\documentclass{amsart}

\usepackage[alphabetic,initials,nobysame]{amsrefs}
\usepackage{amssymb,mathtools,mathrsfs,enumitem,hyperref,color}
\usepackage{pgfplots} \pgfplotsset{compat=1.18}
\usetikzlibrary{quotes,calc}

\BibSpec{collection.article}{%
	+{}  {\PrintAuthors}				{author}
	+{,} { \textit}                  	{title}
	+{.} { }                         	{part}
	+{:} { \textit}                  	{subtitle}
	+{,} { \PrintContributions}      	{contribution}
	+{,} { \PrintConference}         	{conference}
	+{}  {\PrintBook}                	{book}
	+{,} { }                         	{booktitle}
	+{,} { }						 	{series}
	+{,} { }						 	{publisher}
	+{,} { }						 	{address}
	+{,} { \PrintDateB}              	{date}
	+{,} { pp.~}                     	{pages}
	+{,} { }                         	{status}
	+{,} { \PrintDOI}                	{doi}
	+{,} { available at \eprint}     	{eprint}
	+{}  { \parenthesize}            	{language}
	+{}  { \PrintTranslation}        	{translation}
	+{;} { \PrintReprint}            	{reprint}
	+{.} { }                         	{note}
	+{.} {}                          	{transition}
	+{}  {\SentenceSpace \PrintReviews} {review}
}

\theoremstyle{plain}
\newtheorem{theorem}{Theorem}
\newtheorem{lemma}[theorem]{Lemma}
\newtheorem{prop}[theorem]{Proposition}
\newtheorem{corollary}[theorem]{Corollary}

\theoremstyle{definition}
\newtheorem{definition}[theorem]{Definition}

\theoremstyle{remark}
\newtheorem{remark}[theorem]{Remark}
\newtheorem{question}[theorem]{Question}

\numberwithin{equation}{section}
\numberwithin{theorem}{section}

\newcommand{\br}{\overline}
\newcommand{\R}{\mathbb R}

\newcommand{\C}{\mathbb C}

\newcommand{\N}{\mathbb N}

\DeclareMathOperator{\dist}{{\mathrm{dist}}}
\DeclareMathOperator{\diam}{{\mathrm{diam}}}

\DeclareMathOperator{\inter}{{\mathrm{int}}}

\DeclareMathOperator{\Mod}{\mathrm{Mod}}

\DeclareMathOperator{\area}{\mathrm{Area}}
\DeclareMathOperator{\loc}{\mathrm{loc}}

\DeclareMathOperator{\st}{\mathrm{St}}
\DeclareMathOperator{\Ast}{\mathcal A\text{-}St}
\DeclareMathOperator{\A'st}{\mathcal{A}^{\prime}\text{-}St}
\DeclareMathOperator{\llc}{\mathit{LLC}}

\title{Characterization of quasispheres via smooth approximation}
\author{Dimitrios Ntalampekos}
\address{Department of Mathematics, Aristotle University of Thessaloniki, Thessaloniki, 54152, Greece.}
\email{dntalam@math.auth.gr}
\date{\today}
\begin{document}

	\begin{abstract}
	We prove that every two-dimensional quasisphere is the limit of a sequence of smooth spheres that are uniform quasispheres. In the case of metric spheres of finite area we provide necessary and sufficient geometric conditions for a quasisphere, involving the doubling property, linear local connectivity, the Loewner property, conformal modulus, and reciprocity. In particular, although an arbitrary quasisphere does not satisfy necessarily all of those geometric conditions, we prove that every quasisphere can be approximated by uniform quasispheres that are uniformly doubling, linearly locally connected, $2$-Loewner, reciprocal and satisfy a uniform bound for the modulus of annuli.
	\end{abstract}

\maketitle

\setcounter{tocdepth}{1}
\tableofcontents

\section{Introduction}

\subsection{Background}
In this paper we provide characterizations of quasispheres involving geometric conditions and approximation by smooth spaces. For $n\geq 2$, an $n$-dimensional {quasisphere} is a metric space that can be mapped to the Euclidean $n$-sphere with a quasisymmetric homeomorphism (see Section \ref{section:prelim} for the definition). The class of quasisymmetric maps generalizes conformal homeomorphisms between Euclidean domains and can be defined on arbitrary metric spaces, even without the presence of a smooth structure. 

Quasispheres appear naturally in the areas of Analysis on Metric Spaces and Geometric Function Theory, but also in adjacent areas. For example, in Complex Dynamics they appear in the study of expanding Thurston maps. Specifically, it is shown in \cite{BonkMeyer:Thurston} that for an expanding Thurston map $f$ of the $2$-sphere, the associated visual metric on the $2$-sphere is quasisymmetric to the Euclidean metric if and only if the map $f$ is topologically conjugate to a rational map. Furthermore, a major open conjecture in Geometric Group Theory, namely Cannon's conjecture \cites{Sullivan:Cannon,Cannon:Conjecture,BonkKleiner:cannon}, asserts that the boundary at infinity of a Gromov hyperbolic group is always a quasisphere if it is topologically a $2$-sphere. Hence, there is a considerable amount of interest in understanding completely these objects, at least in the $2$-dimensional setting.

The most significant work on quasispheres so far is due to Bonk and Kleiner \cite{BonkKleiner:quasisphere}, who characterized Ahlfors $2$-regular quasispheres as \textit{linearly locally connected}, or else $\llc$, spheres (see Section \ref{section:geometric_notions}). Roughly speaking, this condition prevents cusps and dense wrinkles on a surface. The Bonk--Kleiner theorem has been an object of intense study over the past two decades and has been extended to other topologies \cites{Wildrick:parametrization, MerenkovWildrick:uniformization, GeyerWildrick:uniformization, Ikonen:isothermal, HakobyanRehmert:koebe}.

While the $\llc$ condition is necessary for a quasisphere, Ahlfors $2$-regularity is too restrictive. In fact, quasispheres can even have infinite area. Currently, we have very limited knowledge about quasispheres that are not Ahlfors $2$-regular or have infinite area, such as the aforementioned instances in Complex Dynamics and Geometric Group Theory. 

We start our investigation from smooth quasispheres, which are very well understood from a quantitative point of view. With the current techniques it is not hard to establish the following result (actually we prove a more general result, Theorem \ref{theorem:reciprocal} below). We refer the reader to Section \ref{section:geometric_notions} for the definitions of the metric notions appearing in the statement.

\begin{theorem}[Smooth quasispheres]\label{theorem:smooth}
Let $X$ be a Riemannian $2$-sphere. The following statements are quantitatively equivalent.
\begin{enumerate}[label=\normalfont{(\Alph*)}]
\item\label{smooth:quasisphere} $X$ is a quasisphere.
\item\label{smooth:llc-mod}
\begin{enumerate}[label=\normalfont{(\arabic*)}]
\item\label{smooth:llc} $X$ is a doubling and linearly locally connected space and
\item\label{smooth:modulus} there exist constants $L>1$ and $M>0$ such that for every ball $B(a,r)\subset X$ we have 
$$\Mod_2 \Gamma(\br B(a,r), X\setminus B(a,L r);X) <M.$$
\end{enumerate} 
\item\label{smooth:loewner} $X$ is a doubling and $2$-Loewner space. 
\end{enumerate}
\end{theorem}

Of course a smooth sphere satisfies \ref{smooth:llc-mod} and \ref{smooth:loewner} trivially, since it is locally bi-Lipschitz to the Euclidean plane. So the content of the theorem is about the quantitative relation of the quasisymmetric distortion function with the various parameters in \ref{smooth:llc-mod} and \ref{smooth:loewner}. Moreover, condition \ref{smooth:llc-mod}\ref{smooth:modulus} is a consequence of Ahlfors $2$-regularity. We remark that Semmes \cite{Semmes:chordarc2}*{Theorem 5.4} had already provided sufficient conditions in the spirit of Bonk--Kleiner for a smooth $2$-sphere to be a quasisphere, quantitatively. In this paper we investigate the following question.

\begin{question}
To what extent does the conclusion of Theorem \ref{theorem:smooth} remain valid for arbitrary, non-smooth, metric $2$-spheres?
\end{question}

We explore answers to this question in two different directions. First, we discuss the case of spheres of finite area and attempt to provide a characterization in the spirit of Theorem \ref{theorem:smooth}. Second, and in fact as our main theorem, we prove that every quasisphere of dimension $2$ can be approximated by smooth uniform quasispheres. We conclude that, although an arbitrary $2$-dimensional quasisphere does not satisfy \ref{smooth:llc-mod}, it can be approximated by smooth spheres that satisfy this condition with uniform parameters.

\subsection{Quasispheres of finite area}

We show that the implications \ref{smooth:llc-mod} $\Rightarrow$ \ref{smooth:quasisphere} $\Rightarrow$ \ref{smooth:loewner} in Theorem \ref{theorem:smooth} remain true for all metric $2$-spheres of finite area. This is a consequence of the very recent uniformization result of the author and Romney \cites{NtalampekosRomney:nonlength}, which implies that every metric $2$-sphere of finite area can be parametrized by the Euclidean $2$-sphere with a {weakly quasiconformal map}; see Section \ref{section:finite_area} for the definition. This result generalizes the classical uniformization theorem for smooth surfaces and is the final result in a series of recent works on uniformization of metric $2$-spheres of finite area under minimal geometric assumptions, starting from the Bonk--Kleiner theorem  \cites{BonkKleiner:quasisphere, Rajala:uniformization, LytchakWenger:parametrizations, MeierWenger:uniformization, NtalampekosRomney:length, NtalampekosRomney:nonlength}.

\begin{theorem}\label{theorem:bac}
Let $X$ be a metric $2$-sphere of finite Hausdorff $2$-measure. Then the implications \ref{smooth:llc-mod} $\Rightarrow$ \ref{smooth:quasisphere} $\Rightarrow$ \ref{smooth:loewner} in Theorem \ref{theorem:smooth} are true, quantitatively.
\end{theorem}

In addition, we prove that the converse implications remain valid for \textit{reciprocal spheres}, which were introduced by Rajala \cite{Rajala:uniformization}. These are precisely spheres of finite area that can be mapped to the Euclidean $2$-sphere with a quasiconformal map (according to the geometric definition involving modulus). In particular, all Riemannian and polyhedral $2$-spheres are reciprocal by the classical uniformization theorem. 

\begin{theorem}\label{theorem:reciprocal}
Let $X$ be a metric $2$-sphere of finite Hausdorff $2$-measure that is reciprocal. Then the conclusion of Theorem \ref{theorem:smooth} is true.
\end{theorem}

See \cite{Rehmert:thesis}*{Theorem 1.5.7} for a relevant result in the setting of metric spaces homeomorphic to multiply connected domains in the $2$-sphere. On the other hand, for general metric $2$-spheres of finite area the implication \ref{smooth:quasisphere} $\Rightarrow$ \ref{smooth:llc-mod} fails, as a consequence of the next theorem. 

\begin{theorem}\label{theorem:qc_qs}
Let $X$ be a metric $2$-sphere of finite Hausdorff $2$-measure and $h\colon \widehat{\C}\to X$ be a quasisymmetric homeomorphism. The following are quantitatively equivalent.
\begin{enumerate}[label=\normalfont{(\arabic*)}]
	\item\label{qc_qs:1} The map $h$ is quasiconformal.
	\item\label{qc_qs:2} There exist constants $L>1$ and $M>0$ such that for every $a\in X$ we have
$$\liminf_{r\to 0^+}\Mod_2 \Gamma(\br B(a,r), X\setminus B(a,L r);X) <M.$$
	\item\label{qc_qs:3} There exist constants $L'>1$ and $M'>0$ such that for every ball $B(a,r)\subset X$ we have
$$\Mod_2 \Gamma(\br B(a,r), X\setminus B(a,L'r);X) <M'.$$
\end{enumerate}
\end{theorem}

Here $\widehat \C$ denotes the Riemann sphere equipped with the spherical metric and measure. This result implies that a quasisphere of finite area that satisfies condition \ref{smooth:llc-mod}\ref{smooth:modulus} of Theorem \ref{theorem:smooth} is necessarily reciprocal. On the other hand, examples of quasispheres of finite area that are not reciprocal have been presented by Romney and the author \cites{Romney:absolute, NtalampekosRomney:absolute}. In particular, there exists a quasisphere $X$ with finite area in which condition \ref{smooth:llc-mod}\ref{smooth:modulus} fails. 

We are not aware whether the implication \ref{smooth:loewner} $\Rightarrow$ \ref{smooth:quasisphere} fails in general.
\begin{question}
Does there exist a doubling and $2$-Loewner metric sphere that is not a quasisphere? What if we add further assumptions, such as the $\llc$ property?
\end{question}
If such a space exists, then by Theorem \ref{theorem:reciprocal} it cannot be reciprocal, so it does not admit a quasiconformal parametrization from the Euclidean sphere, but only a weakly quasiconformal one. We prove Theorems \ref{theorem:bac}, \ref{theorem:reciprocal}, and \ref{theorem:qc_qs} in Section \ref{section:finite_area}.

\subsection{Approximation by smooth quasispheres}

A typical example of a quasisphere that is not Ahlfors $2$-regular and has infinite area is the \textit{snowsphere} of Meyer \cite{Meyer:origami}; see Figure \ref{figure:snowsphere}. The snowsphere is constructed as follows. Consider the surface of the unit cube in $\R^3$. Each face is a unit square and we subdivide it into nine squares of side length $1/3$. Then we replace the middle square with a cubical cap, consisting of five faces of side length $1/3$. We repeat this subdivision and replacement process in each square of side length $1/3$, etc. The space that we obtain in each stage of the construction is a polyhedral surface with its intrinsic metric and it is actually a quasisphere with uniform parameters. The snowsphere is the Gromov--Hausdorff limit of that sequence of spaces. Motivated by the example of the snowpshere we pose the following question.

\begin{figure}
\centering
\includegraphics[scale=0.5]{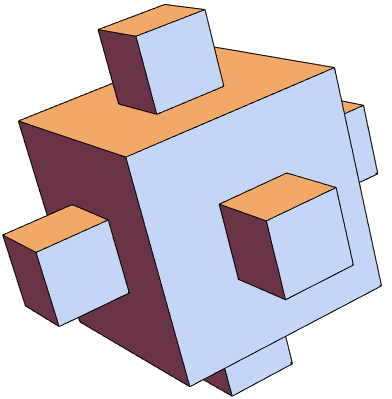}
\caption{The first stage of the construction of the snowsphere.}\label{figure:snowsphere}
\end{figure}

\begin{question}\label{question:approximation}
Does every quasisphere arise as the limit of a sequence of polyhedral or smooth spheres that are uniform quasispheres?
\end{question}

As the main result of the paper, we provide an affirmative answer to the question. 

\begin{theorem}[Smooth approximation of quasisymmetric $2$-manifolds]\label{theorem:main:approximation}
Let $X$ be a compact Riemannian $2$-manifold (with boundary), equipped with a Riemannian metric $g$ and with the corresponding intrinsic metric $d_g$. Let $d_X$ be another metric on $X$ that induces its topology. The following are quantitatively equivalent.
	\begin{enumerate}[label=\normalfont(\arabic*)]
	\item\label{theorem:main:approximation:1} The metric space $(X,d_g)$ is quasisymmetric to $(X,d_X)$.
	\item\label{theorem:main:approximation:2}  For each $k\in \N$ there exists a Riemannian metric $g_k$ on $X$ and a metric $d_k$ that is locally isometric to the intrinsic metric $d_{g_k}$ such that 
	\begin{align*}
	&\text{$(X,d_k)$ converges to $(X,d_X)$ in the Gromov--Hausdorff metric as $k\to\infty$ and}\\
	&\text{$(X,d_{k})$ is uniformly quasisymmetric to $(X,d_g)$ for each $k\in \N$.}
	\end{align*}
	\end{enumerate}
In this case there exists an approximately isometric sequence $f_k\colon (X,d_k) \to (X,d_X)$, $k\in \N$, of uniform quasisymmetries.
\end{theorem}

The implication from \ref{theorem:main:approximation:2} to \ref{theorem:main:approximation:1} is quite standard. We reformulate the reverse implication, which is the most technical, in a way that the quantitative dependence is more clear. Let $X$ be a smooth manifold. We define $R_{\loc}(X)$ to be the collection of metrics $d$ on $X$ that are locally isometric to the intrinsic metric arising from a Riemannian metric on $X$. For a metric $d$ on $X$ and a homeomorphism $\eta\colon [0,\infty)\to [0,\infty)$ we denote by $QS(X,d,\eta)$ the collection of metric spaces $Y$ with the property that there exists an $\eta$-quasisymmetric map from $(X,d)$ onto $Y$. 
The implication from \ref{theorem:main:approximation:1} to \ref{theorem:main:approximation:2} in our main theorem can be restated as follows. Note that the closure refers to the Gromov--Hausdorff distance.

\begin{theorem}\label{theorem:main:approximation:rloc}
For each distortion function $\eta$ there exists a distortion function $\eta'$ such that for each compact Riemannian $2$-manifold $X$ with intrinsic metric $d_g$ we have
$$QS(X,d_g,\eta)\subset \overline{ R_{\loc}(X)\cap QS(X,d_g,\eta')}.$$
\end{theorem}

We remark that Bonk and Kleiner, in their main theorem \cite{BonkKleiner:quasisphere}*{Theorem 11.1}, prove that each $2$-dimensional quasisphere admits good graph approximations that satisfy condition \ref{smooth:llc-mod}\ref{smooth:modulus} with discrete modulus in place of conformal modulus. The definition of those graph approximations as well as the discrete modulus condition are technical to state and handle. In contrast, conformal modulus can be defined directly on a surface of finite area using the Hausdorff $2$-measure or equivalently (if the surface is smooth) in local coordinates and is perhaps a more tangible object than discrete modulus on graph approximations of a space. We remark that conformal modulus and discrete modulus are not comparable in general, even in the setting of quasispheres of finite area. This is illustrated by the fact that not every quasisphere of finite area satisfies \ref{smooth:llc-mod}\ref{smooth:modulus}, as already discussed, but as a consequence of the result of Bonk and Kleiner the discrete analogue of that condition is always satisfied. Hence, our main theorem is not a consequence of \cite{BonkKleiner:quasisphere}*{Theorem 11.1}.

On the other hand, as a corollary of Theorem \ref{theorem:main:approximation} and Theorem \ref{theorem:reciprocal}, one can always approximate a quasisphere by smooth quasispheres that satisfy the conformal modulus analogue \ref{smooth:llc-mod}\ref{smooth:modulus} of the condition of Bonk--Kleiner.

\begin{corollary}
Every $2$-dimensional quasisphere is the Gromov--Hausdorff limit of a sequence of $2$-dimensional quasispheres that are doubling, linearly locally connected, $2$-Loewner, reciprocal, and satisfy condition \ref{smooth:llc-mod}\ref{smooth:modulus}, quantitatively.
\end{corollary}

Theorem \ref{theorem:main:approximation} is similar in spirit to the main results in \cites{NtalampekosRomney:length, NtalampekosRomney:nonlength}, which assert that every metric surface of locally finite area can be approximated in the Gromov--Hausdorff sense by polyhedral surfaces of controlled geometry. This result has been generalized to higher dimensions by Marti and Soultanis under some additional geometric assumptions \cite{MartiSoultanis:metric_fundamental_class}. Our argument for Theorem \ref{theorem:main:approximation} is very robust and we expect that if one imposes further properties on a quasisphere, such as Ahlfors $2$-regularity or finiteness of area, then the smooth approximating surfaces will also have these properties.

\subsection{Bi-Lipschitz surfaces}
We include a discussion on metric surfaces that are bi-Lipschitz to smooth surfaces. The analogues of Theorem \ref{theorem:main:approximation} and Theorem \ref{theorem:main:approximation:rloc} remain valid if one replaces quasisymmetric with bi-Lipschitz maps. Let $(X,d)$ be a metric space. For $\lambda\geq 1$ we denote by $BL(X,d,\lambda)$ the collection of metric spaces $Y$ with the property that there exists a $\lambda$-bi-Lipschitz map from $(X,d)$ onto $Y$. 

\begin{theorem}\label{theorem:main:approximation:bilip}
For each $\lambda\geq 1$ there exists $\lambda'\geq 1$ such that for each compact Riemannian $2$-manifold $X$  with intrinsic metric $d_g$ we have
$$BL(X,d_g,\lambda)\subset \overline{ R_{\loc}(X)\cap BL(X,d_g,\lambda')}.$$
\end{theorem}

The proofs of Theorems \ref{theorem:main:approximation} and \ref{theorem:main:approximation:bilip} follow exactly the same scheme, but the verification of the bi-Lipschitz property is far more elementary than the quasisymmetric property. 

Although a quantitative characterization of smooth quasispheres is available thanks to Theorem \ref{theorem:smooth}, this is not the case with smooth bi-Lipschitz spheres, i.e., Riemannian $2$-spheres that are bi-Lipschitz equivalent to the Euclidean $2$-sphere. So far there are some sufficient conditions so that a smooth sphere or plane is bi-Lipschitz equivalent to the Euclidean sphere or plane, respectively. For example, Fu \cite{Fu:biLipschitz} and Bonk--Lang \cite{BonkLang:biLipschitz} prove that certain bounds on the integral Gaussian curvature provide a sufficient condition for quantitative bi-Lipschitz parametrization. The result of Fu strengthens an earlier result of Toro 
\cites{Toro:LipschitzManifolds, Toro:biLipschitz}, implying that graphs of functions in the Sobolev space $W^{2,2}(\R^2)$ admit local bi-Lipschitz parametrizations, quantitatively, depending on the $W^{2,2}$ norm. See also the related works \cites{Semmes:hypersurfaces, MullerSverak:surfaces}. 

\subsection{Proof sketch}
The proofs of Theorems \ref{theorem:main:approximation} and \ref{theorem:main:approximation:bilip} are given in Section \ref{section:proof}. We present here a sketch of the proof of the main result, Theorem \ref{theorem:main:approximation}. We describe the basic steps and accompany them with an illustration of the argument in Figure \ref{figure:sketch}. The first step is to triangulate the smooth surface $(X,d_g)$. By a result of Bowditch \cite{Bowditch:triangulation}, one can find a polyhedral surface $(Z,d_Z)$ consisting of equilateral triangles of equal side length and a (uniformly) bi-Lipschitz map $\tau\colon (Z,d_Z)\to (X,d_g)$; this is where the compactness of $X$ is used. The size of the triangles can be taken to be arbitrarily small.

The (smooth) triangulation of $(X,d_g)$ gives a (fractal) triangulation of $(X,d_X)$ via the identity map, which is assumed to be quasisymmetric. The idea of the proof is to replace each fractal triangle of $(X,d_X)$ with a polyhedral surface that is quasisymmetric to the original triangle. All polyhedral surfaces considered in the proof have a nice geometric structure in the sense that they consist of triangles, each of which must be uniformly bi-Lipschitz to an equilateral triangle and a uniformly bounded number of triangles can meet at a point. Such objects are termed quasiconformal simplicial complexes and we study them in detail in Section \ref{section:simplicial}.

\begin{figure}
\centering
\input{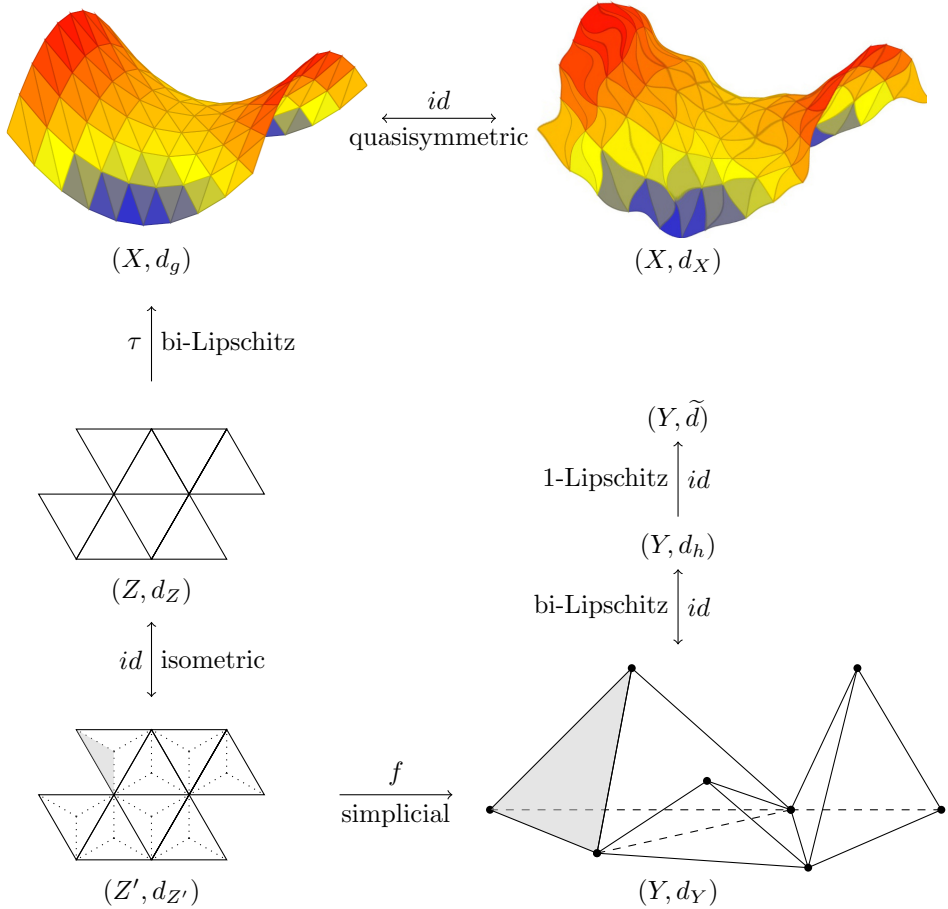}
\caption{Sketch of the proof of Theorem \ref{theorem:main:approximation}.}\label{figure:sketch}
\end{figure}

In order to construct those simplicial complexes that are going to replace the fractal triangles of $(X,d_X)$, we work with the triangles of $(Z,d_Z)$. Suppose that a triangle $S$ of $(Z,d_Z)$ has vertices $p_1,p_2,p_3$. The corresponding fractal triangle in $(X,d_X)$ has vertices $\tau(p_1),\tau(p_2),\tau(p_3)$. Note that one can construct in the plane a possibly degenerate triangle with side lengths $d_X(\tau(p_i),\tau(p_j))$, $i\neq j$. Using that triangle one can construct a quasiconformal simplicial complex $K_S$ homeomorphic to a disk, whose boundary consists of three line segments of side lengths $d_X(\tau(p_i),\tau(p_j))$, $i\neq j$. Pasting together those simplicial complexes gives a simplicial map $f$ (i.e., linear on triangles) from a suitable subdivision $(Z',d_{Z'})$ of $(Z,d_Z)$ onto a quasiconformal simplicial complex $(Y,d_Y)$ as shown in Figure \ref{figure:sketch}. 

The next step is to find a uniformly bi-Lipschitz map from the simplicial complex $(Y,d_Y)$ onto a smooth Riemannian surface $(Y,d_h)$. This is achieved by a recent result of Cattalani \cite{Cattalani:smoothing}, which allows the bi-Lipschitz smoothing of quasiconformal simplicial complexes, quantitatively. 

Finally, the last and most technical step is to bring the surface $(Y,d_h)$ close to $(X,d_X)$ while ensuring at the same time that the constructed surface remains quasisymmetric to $(X,d_g)$ with uniform parameters. Note that $(Y,d_h)$ is a length space and there is no reason for it to be close to the space $(X,d_X)$, which might be very far from being a length space. 

In order to bring the space $(Y,d_h)$ close to $(X,d_X)$ in the Gromov--Hausdorff sense, one can simply change the distance between the vertices of $(Y,d_h)$ and declare it to be equal to the distance between the corresponding vertices of $(X,d_X)$. This gives a metric space $(Y,\widetilde d)$ that is locally isometric to the smooth space $(Y,d_h)$, so in a sense $(Y,\widetilde d)$ is a smooth space. See Lemma \ref{lemma:glue} for the precise procedure of changing the metric on vertices. What remains to be done is to show that $f\circ \tau^{-1}$ gives a uniformly quasisymmetric map from the original Riemannian surface $(X,d_g)$ onto the constructed smooth surface $(Y,\widetilde d)$. This is achieved with the aid of Theorem \ref{theorem:approximation_quasisymmetric}, which is formulated in the language of approximations of metric spaces, a notion originally introduced by Bonk and Kleiner. 

Theorem \ref{theorem:approximation_quasisymmetric} is the most elaborate result of the paper and provides quantitative control of the quasisymmetric distortion function $\eta'$ of the resulting map $f\circ \tau^{-1}\colon (X,d_g)\to (Y,\widetilde d)$. Remarkably, and as illustrated more clearly in Theorem \ref{theorem:main:approximation:rloc}, the distortion function $\eta'$ depends only on the distortion function $\eta$ of the quasisymmetric map $(X,d_g)\to (X,d_X)$ and not on the geometry of the compact surface $(X,d_g)$, which could be uncontrolled; for example, $(X,d_g)$ could be a doubling or $\llc$ metric space with bad (i.e., quite large) parameters.

\subsection*{Acknowledgments}
I would like to thank Daniel Meyer for a motivating discussion on the topic and for posing Question \ref{question:approximation} during the Quasiworld Workshop (Helsinki, August 2023), which gave rise to this project.

\section{Approximations of metric spaces and quasisymmetries}\label{section:approximation}
In this section we introduce a variant of an approximation of a metric space, as defined by Bonk and Kleiner \cite{BonkKleiner:quasisphere}*{Section 4}. Roughly speaking, an approximation of a metric space is a graph on the space with controlled combinatorial and metric properties. The main result of the section is Theorem \ref{theorem:approximation_quasisymmetric} and it provides sufficient conditions so that a homeomorphism between metric spaces that respects a pair of approximations of those spaces is quasisymmetric, quantitatively. This section can be read independently of the other parts of the paper.

\subsection{Preliminaries}\label{section:prelim}
For quantities $A$ and $B$ we write $A\lesssim B$ if there exists a constant $c>0$ such that $A\leq cB$. If the constant $c$ depends on another quantity $H$ that we wish to emphasize, then we write instead $A\leq c(H)B$ or $A\lesssim_H B$. Moreover, we use the notation $A\simeq B$ if $A\lesssim B$ and $B\lesssim A$. As previously, we write $A\simeq_H B$ to emphasize the dependence of the implicit constants on the quantity $H$. All constants in the statements are assumed to be positive even if this is not stated explicitly and the same letter may be used in different statements or within the same proof to denote a different constant.  

Let $(X,d)$ be a metric space. We denote by $B_d(x,r)$ (resp.\ $\br{B_d}(x,r)$) the open (resp.\ closed) ball centered at $x$ with radius $r>0$. For a set  $A\subset X$ and $r>0$ we denote by $N_d(A,r)$ the open $r$-neighborhood of $A$ and by $\diam_d(A)$ the diameter of $A$. Often we drop the subscript $d$ from the above notation when this does not lead to a confusion.

Let $(X,d_X)$ and $(Y,d_Y)$ be metric spaces and $f\colon X\to Y$ be a homeomorphism. We say that $f$ is \textit{quasisymmetric} if there exists a homeomorphism $\eta\colon[0,\infty)\to [0,\infty)$ such that for every triple of distinct points $x,y,z\in X$ we have
$$\frac{d_Y(f(x),f(y))}{d_Y(f(x),f(z))}\leq \eta\left(\frac{d_X(x,y)}{d_X(x,z)}\right).$$
In that case we say that $f$ is $\eta$-quasisymmetric. A homeomorphism $\eta\colon [0,\infty)\to[0,\infty)$ as above is called a \textit{distortion function}. A homeomorphism $f\colon X \to Y$ is \textit{bi-Lipschitz} if there exists $\lambda\geq1$ such that 
$$\lambda^{-1}d_X(x,y)\leq d_Y(f(x),f(y))\leq \lambda d_X(x,y)$$
for every $x,y\in X$. A metric space $X$ is called an \textit{$n$-dimensional quasisphere}, where $n\geq1$, if there exists a quasisymmetric map $f\colon \mathbb S^n\to X$, where $\mathbb S^n$ denotes the Euclidean $n$-dimensional unit sphere in $\R^{n+1}$.

\begin{lemma}[\cite{Heinonen:metric}*{Proposition 10.8}]\label{lemma:qs_heinonen}
Let $X,Y$ be metric spaces, $\eta\colon [0,\infty)\to [0,\infty)$ be a homeomorphism, and $f\colon X\to Y$ be an $\eta$-quasisymmetric homeomorphism. If $A\subset B\subset X$ and $0<\diam(A)\leq \diam(B)<\infty$, then 
\begin{align*}
\frac{1}{2\eta\left(\frac{\diam(B)}{\diam(A)}\right)} \leq \frac{\diam(f(A))}{\diam(f(B))}\leq \eta \left(\frac{2\diam(A)}{\diam(B)}\right).
\end{align*}
\end{lemma}

\subsection{Gluing metrics}
A metric space $(Y,d_Y)$ is called \textit{discrete} if for each $x\in Y$ there exists $r>0$ such that $B(x,r)=\{x\}$. The next result describes the process of gluing a new  metric in a subset of a given metric space $(Y,d_Y)$.

\begin{lemma}\label{lemma:glue}
Let $(Y,d_Y)$ be a metric space and $S\subset Y$ be a closed set with a metric $d_S$ such that $d_S\leq d_Y$ on $S\times S$. For $x,y\in Y$, define
\begin{align*}
\widetilde d(x,y)= \min\{d_Y(x,y), \inf_{u,v\in S} \{d_Y(x,u)+d_S(u,v)+d_Y(v,y)\}\}.
\end{align*}
The following statements are true. 
\begin{enumerate}[label=\normalfont{(\arabic*)}]
	\item\label{g:metric} $(Y,\widetilde d)$ is a metric space.
	\item\label{g:cases} If $x,y\in S$, then $\widetilde d(x,y)=d_S(x,y)$. If $x\in Y\setminus S$ and $y\in S$, then
	$$\widetilde d(x,y)=\inf_{u\in S}\{d_Y(x,u)+d_S(u,y)\}.$$ 
	\item\label{g:isom} For each $x,y\in Y$ and $r=\dist_{\widetilde d}(x,S\setminus \{x\})$, if $\widetilde d(x,y)< r$, then $\widetilde d(x,y)=d_Y(x,y)$. In addition,  the identity map restricts to an isometric homeomorphism between $B_{d_Y}(x,r/3)$ and $B_{\widetilde d}(x,r/3)$.
	\item\label{g:homeo} If $(S,d_S)$ is topologically equivalent to $(S,d_Y)$, then the identity map from $(Y,d_Y)$ onto $(Y,\widetilde d)$ is a $1$-Lipschitz homeomorphism.
	\item\label{g:local_isom}If $(S,d_S)$ is a discrete metric space, then the identity map from $(Y,d_Y)$ onto $(Y,\widetilde d)$ is a local isometry.
	\item\label{g:lambda}If $d_S\geq \lambda^{-1} d_Y$ on $S\times S$ for some $\lambda\geq 1$, then $\lambda^{-1}d_Y\leq \widetilde d\leq d_Y$.
\end{enumerate}
\end{lemma}
We say that $\widetilde d$ is the \textit{metric arising from gluing $d_Y$ with $d_S$}.

\begin{proof}
Note that $\widetilde d$ is clearly symmetric. If $\widetilde d(x,y)=0$, then either $d_Y(x,y)=0$, so $x=y$, or there exist points $u_n,v_n\in S$, $n\in \N$, such that $d_Y(x,u_n)+d_S(u_n,v_n)+d_Y(v_n,y)\to 0$ as $n\to\infty$. The assumption that $S$ is closed in $(Y,d_Y)$ implies that $x,y\in S$.  Also, since $d_S\leq d_Y$, we have $u_n\to x$ in $d_S$ and $v_n\to y$ in $d_S$. Since $d_S(u_n,v_n)\to 0$, we conclude that $d_S(x,y)=0$, so $x=y$.  Finally, for the triangle inequality, let $x,y,z\in Y$ and $s,t,u,v\in S$. Since $d_S\leq d_Y$ on $S\times S$, we have
\begin{align*}
\widetilde d(x,y)&\leq d_Y(x,u)+d_S(u,v)+d_Y(v,y)\\
&\leq d_Y(x,u)+d_S(u,s)+ d_S(s,t)+d_S(t,v)+d_Y(v,y)\\
&\leq (d_Y(x,u)+d_S(u,s)+d_Y(s,z)) + (d_Y(z,t)+d_S(t,v)+d_Y(v,y)).
\end{align*}
Also, by the definition of $\widetilde d$, one obtains immediately the estimates
\begin{align*}
\widetilde d(x,y) &\leq  d_Y(x,z) + (d_Y(z,t)+d_S(t,v)+d_Y(v,y)),\\
\widetilde d(x,y) &\leq  (d_Y(x,u)+d_S(u,s)+d_Y(s,z)) + d_Y(z,y),\\
\widetilde d(x,y)&\leq d_Y(x,y)\leq d_Y(x,z)+d_Y(z,y)
\end{align*}
Taking infimum over all $s,t,u,v$ gives the triangle inequality, completing the proof of \ref{g:metric}.

Now, we prove \ref{g:cases}. If $x,y\in S$, then for every $u,v\in S$ we have
\begin{align*}
d_S(x,y) \leq d_S(x,u)+d_S(u,v)+d_S(v,y)\leq d_Y(x,u)+d_S(u,v)+d_Y(v,y),
\end{align*}
because $d_S\leq d_Y$. This shows that $\widetilde d(x,y)\geq d_S(x,y)$. Moreover, equality holds for $u=x$ and $v=y$. Thus, $\widetilde d(x,y)=d_S(x,y)$. Now, let $x\in Y\setminus S$ and $y\in S$, as in the second part of \ref{g:cases}. It is clear, by the definition of the infimum, that
\begin{align*}
d_Y(x,y)&\geq \inf_{u\in S} \{d_Y(x,u)+d_S(u,y)\} \\
&\geq  \inf_{u,v\in S} \{d_Y(x,u)+d_S(u,v)+d_Y(v,y)\}=\widetilde d(x,y).
\end{align*}
On the other hand, for every $u,v\in S$ we have
$$d_Y(x,u)+d_S(u,v)+d_Y(v,y)\geq d_Y(x,u)+d_S(u,v)+d_S(v,y)\geq d_Y(x,u)+d_S(u,y).$$ 
This shows that 
$$\widetilde d(x,y)\geq \inf_{u\in S} \{d_Y(x,u)+d_S(u,y)\},$$
which completes the proof of \ref{g:cases}.

Next, we establish \ref{g:isom}. If $x\in Y\setminus S$, let $r=\dist_{\widetilde d}(x,S)=\dist_{\widetilde d}(x,S\setminus \{x\})$. Let $z\in S$. By \ref{g:cases}, for each $\varepsilon>0$ there exists  $u\in S$ such that 
$$\widetilde d(x,z)\geq d_Y(x,u)+ d_S(u,z)-\varepsilon\geq \dist_{d_Y}(x,S) -\varepsilon.$$
This implies that $r\geq \dist_{d_Y}(x,S)>0$. Let $y\in Y$ with $\widetilde d(x,y)< r$. For $u,v\in S$ we have
$$d_Y(x,u)+d_S(u,v)+d_Y(v,y) \geq \widetilde d(x,u)\geq r >\widetilde d(x,y).$$
The definition of $\widetilde d$ implies that $\widetilde d(x,y)=d_Y(x,y)$. Now, if  $y,z\in B_{\widetilde d}(x,r/3)$, then $\dist_{\widetilde d}(y,S)\geq 2r/3 > \widetilde d(y,z)$. Hence, by what we have proved, $\widetilde d(y,z)=d_Y(y,z)$, which implies that $y,z\in B_{d_Y}(x,r/3)$. Also, $B_{d_Y}(x,r/3)\subset B_{\widetilde d}(x,r/3)$ trivially.  

If $x\in S$, let $r=\dist_{\widetilde d}(x,S\setminus \{x\})$. Note that $r$ might be zero, in which case we have nothing to prove. Suppose that $r>0$ and let $y\in Y$ with $0<\widetilde d(x,y)< r$. In particular, $y\notin S$. For $u\in S\setminus \{x\}$, by \ref{g:cases}, we have
\begin{align*}
d_S(x,u)+d_Y(u,y) \geq  d_S(x,u)=\widetilde d(x,u) \geq r> \widetilde d(x,y).
\end{align*}
The definition of $\widetilde d$ in the second part of \ref{g:cases} implies that $\widetilde d(x,y)=d_Y(x,y)$. If  $y,z\in B_{\widetilde d}(x,r/3)$, then for $u,v\in S\setminus \{x\}$, we have
\begin{align*}
d_Y(y,u)+d_S(u,v)+d_Y(v,z)\geq \widetilde d(y,u)+\widetilde d(v,z)\geq 2r/3> \widetilde d(y,z).
\end{align*}
If $u=v=x$, then the left-hand side is at least $d_Y(x,y)$ by the triangle inequality. Now, the definition of $\widetilde d$ implies that $\widetilde d(y,z)=d_Y(y,z)$ and $y,z\in B_{d_Y}(x,r/3)$. Also, $B_{d_Y}(x,r/3)\subset B_{\widetilde d}(x,r/3)$ trivially.  This completes the proof of \ref{g:isom}.

For \ref{g:homeo}, note that the identity map from $(Y,d_Y)$ to $(Y,\widetilde d)$ is trivially $1$-Lipschitz. We show the continuity of the inverse map. By \ref{g:isom} we see that the identity map is a local homeomorphism in $Y\setminus S$. Let $x\in S$ and $x_n\in Y$, $n\in \N$, be a sequence with $\widetilde d(x_n,x)\to 0$ as $n\to\infty$. By \ref{g:cases}, for each $n\in \N$ there exists $v_n\in S$ such that $d_S(x,v_n)+d_Y(v_n,x_n)\to 0$ as $n\to\infty$. Since $d_S$ is topologically equivalent to $d_Y$, we conclude that $d_Y(x,v_n)\to 0$. Since $d_Y(v_n,x_n)\to 0$, we conclude that $d_Y(x_n,x)\to 0$, as desired.  

Suppose that $(S,d_S)$ is a discrete metric space, as in \ref{g:local_isom}. Then $(S,d_Y)$ is also discrete since $d_S\leq d_Y$. As a consequence, by \ref{g:homeo}, the identity map is a homeomorphism from $(Y,d_Y)$ onto $(Y,\widetilde d)$. This implies that for each $x\in Y$ we have $r=\dist_{\widetilde d}(x,S\setminus \{x\})>0$, since $\dist_{d_Y}(x,S\setminus \{x\})>0$. By \ref{g:isom} we conclude that the identity map restricts to an isometric homeomorphism between $B_{d_Y}(x,r/3)$ and $B_{\widetilde d}(x,r/3)$.

Finally, suppose that $d_S\geq \lambda^{-1}d_Y$ as in \ref{g:lambda}. For $x,y\in Y$ and $u,v\in S$ we have
$$d_Y(x,u)+d_S(u,v)+d_Y(v,y)\geq \lambda^{-1}(d_Y(x,u)+d_Y(u,v)+d_Y(v,y))\geq \lambda^{-1}d_Y(x,y).$$
The definition of $\widetilde d$ gives immediately that $\widetilde d\geq \lambda^{-1}d_Y$.
\end{proof}

\subsection{Approximations of metric spaces}\label{section:approximations:definition}
Following \cite{BonkKleiner:quasisphere}*{Section 4} we define the notion of an approximation of a metric space as follows. Given a graph $G=(V,\sim)$ we denote by $k_G(u,v)$, or simply by $k(u,v)$ when the graph $G$ is implicitly understood, the combinatorial distance of vertices $u,v\in V$; that is, the minimum number of edges in a chain connecting the two vertices. Note that $k(u,v)$ is understood to be $\infty$ if there is no chain of edges connecting $u$ and $v$. Let $(X,d_X)$ be a metric space. We consider quadruples $\mathcal A=(G,\mathfrak{p},\mathfrak{r}, \mathcal{U})$, where $G=(V,\sim)$ is a graph with vertex set $V$, $\mathfrak p$ is a map $\mathfrak{p}\colon V\to X$, $\mathfrak r$ is a map $\mathfrak{r}\colon V\to (0,\infty)$, and $\mathcal U=\{\mathcal U(v): v\in V\}$ is an open cover of $X$. We let $\mathfrak{p}(v)=p_v$, $\mathfrak{r}(v)=r_v$, and $\mathcal U(v)=U_v$ for $v\in V$. For $K>0$ we define the \textit{$K$-star} of a vertex $v\in V$ with respect to $\mathcal A$ as
$$\mathcal A\text{-}\st_K(v)=\bigcup \{ U_u: u\in V,\, k(u,v)<K\}.$$
We will drop the letter $\mathcal A$ and denote the $K$-star by $\st_K(v)$, whenever this does not lead to a confusion. For $K,L\geq 1$, we call the quadruple $\mathcal A$ a \textit{$(K,L)$-approximation of $X$} if the following four conditions are satisfied. 
\begin{enumerate}[label=\normalfont{(A\arabic*)}]
	\item\label{a:1} Every vertex of $G$ has valence at most $K$.
	\item\label{a:2} $B(p_v,r_v)\subset U_v\subset B(p_v,Lr_v)$ for every $v\in V$.
	\item\label{a:3} Let $u,v\in V$. If $u\sim v$, then $U_u\cap U_v\neq \emptyset$ and $L^{-1}r_u\leq r_v\leq Lr_u$. Conversely, if $U_u\cap U_v\neq \emptyset$, then $k(u,v)<K$.
	\item\label{a:4} $N(U_v,{r_v/L})\subset \st_K(v)$ for every $v\in V$.
\end{enumerate}
Recall that $N(A,r)$ denotes the open $r$-neighborhood of the set $A$. The $(K,L)$-approximation $\mathcal A$ of $X$ is called \textit{fine} if $U_v\neq X$ for all $v\in V$.  

Observe that in the above definition the constant $K$ encodes combinatorial information, which therefore remains invariant under homeomorphisms, and the constant $L$ encodes metric information. Note that if $\mathcal A$ is a $(K,L)$-approximation of $X$ and $K'\geq K$, $L'\geq L$, then $\mathcal A$ is also an $(K',L')$-approximation of $X$. We record some immediate further properties of $(K,L)$-approximations.

\begin{enumerate}[label=\normalfont{(A\arabic*)}]\setcounter{enumi}{4}
	\item\label{a:5con} If $X$ is connected, then $G$ is connected.
	\begin{proof}
	This follows from \ref{a:3}.
	\end{proof}
	\item\label{a:6} Let $u,v\in V$. If $k(u,v)\geq 2K$ then $U_u\cap N(U_v,{r_v/L})=\emptyset$.
	\begin{proof}
	Suppose that $U_u\cap  N(U_v,{r_v/L})\neq \emptyset$, so $U_u\cap  \st_K(v)\neq \emptyset$ by \ref{a:4}. Hence, there exists $w\in V$ such that $U_u\cap  U_w\neq \emptyset$ and $k(w,v)<K$. By \ref{a:3}, we have $k(u,v)\leq k(u,w)+k(w,v)<2K$, a contradiction. 
	\end{proof}
	\item\label{a:7} Let $u,v\in V$. If $k(u,v)\geq 2K$ then for each $x\in U_u$ and $y\in U_v$ we have $$C(L)^{-1}d_X(p_u,p_v)\leq  d_X(x,y)\leq C(L)d_X(p_u,p_v).$$ 
	\begin{proof}
	By \ref{a:6}, $d_X(x,y)\geq r_{u}/L$ and $d_X(x,y)\geq r_{v}/L$.  Therefore, by \ref{a:2},
$$d_X(p_u,p_v)\leq Lr_{u}+ d_X(x,y)+Lr_{v} \leq  (2L^2+1) d_X(x,y).$$
On the other hand, we have $U_{u}\cap U_{v}=\emptyset$ by \ref{a:3}. By \ref{a:2} we conclude that $d_X(p_u,p_v)\geq \max\{r_{u},r_{v}\}$. Therefore,
\[d_X(x,y)\leq Lr_{u}+ d_X(p_u,p_v)+Lr_{v} \leq (2L+1)d_X(p_u,p_v).\qedhere\]
	\end{proof}
\end{enumerate}

\subsection{Approximations and connectedness assumptions}
We say that a metric space $X$ has \textit{bounded turning} if there exists a constant $L\geq 1$ such that for each pair $x,y\in X$ there exists a connected set $E\subset X$ that contains $x$ and $y$ with $\diam(E)\leq Ld_X(x,y)$. In this case we say that $X$ has $L$-bounded turning. 

\begin{lemma}[Connected sets and chains of vertices]\label{lemma:approximation_connected}
Let $X$ be a metric space, $K,L\geq 1$, and $\mathcal A=((V,\sim),\mathfrak{p},\mathfrak{r},\mathcal U)$ be a $(K,L)$-approximation of $X$. Let $u,v\in V$ and $E\subset X$ be a connected set with $p_u,p_v\in E$.  Then there exist vertices $u=w_0, w_1, \dots, w_n=v$, $n\geq 0$, with $k(w_{i-1},w_i)<K$ for each $i\in \{1,\dots,n\}$, such that 
$$\diam (\{p_{w_0},\dots,p_{w_n} \})\leq C(L) \diam(E).$$
In particular, if $X$ has $L$-bounded turning, then for every pair of points $p_u,p_v\in \mathfrak{p}(V)$ there exist points $p_{u}=p_{w_0},p_{w_1},\dots,p_{w_n}=p_{v}\in \mathfrak{p}(V)$ with $k(w_{i-1},w_i)<K$ for each $i\in \{1,\dots,n\}$, such that
$$\diam (\{p_{w_0},\dots,p_{w_n} \})\leq C'(L) d_X(p_{u},p_v).$$
\end{lemma}
\begin{proof}
By \ref{a:2} we have $p_u\in U_u$ and $p_v\in U_v$. If $E\subset U_u$, then $U_v\cap U_u\neq \emptyset$ so $k(u,v)<K$ by \ref{a:3}. We have $\diam(\{p_u,p_v\})\leq \diam(E)$, so the conclusion follows in this case. Suppose that $E\not\subset U_u$. By \ref{a:2}, $\diam(E)\geq r_u$. The collection $\{U_w: w\in V,\, U_w\cap E\neq \emptyset\}$ is an open cover of $E$. The connectedness of $E$ implies that there exist vertices $u=w_0,w_1,\dots,w_n=v$, where $n\geq 0$, such that $U_{w_i}\cap E\neq \emptyset$ for $i\in \{0,\dots,n\}$ and  $U_{w_{i-1}}\cap U_{w_i}\neq \emptyset$ for $i\in \{1,\dots,n\}$. By \ref{a:3}, we have $k(w_{i-1},w_i)<K$ for $i\in \{1,\dots,n\}$. 

If $E\subset \st_K(w_i)$ for some $i\in \{0,\dots,n\}$ then $U_{w_j}\cap \st_K(w_i)\neq \emptyset$ for all $j\in \{0,\dots,n\}$. Property \ref{a:3} implies that $k(w_i,w_j)<2K$ and $C(L)^{-1}r_{w_i}\leq r_{w_j}\leq C(L)r_{w_i}$ for all $j\in \{0,\dots,n\}$. In particular, $r_{w_j}\leq C_1(L) r_{u}\leq C_1(L)\diam (E)$ for all $j\in \{0,\dots,n\}$. By \ref{a:2}, we have $\diam(U_{w_j}) \leq 2Lr_{w_j}\leq C_2(L)\diam(E)$ for $j\in \{0,\dots,n\}$.

Next, suppose that $E\not\subset \st_K(w_i)$ for all $i\in \{0,\dots,n\}$. By \ref{a:4}, since $E\cap U_{w_i}\neq \emptyset$, we have  $\diam(E)\geq r_{w_i}/L$. We conclude that $\diam(U_{w_i})\leq 2Lr_{w_i} \leq 2L^2 \diam(E)$ for all $i\in \{0,\dots,n\}$.	

Combining both cases, we arrive at the conclusion $\diam(U_{w_i})\leq C_3(L) \diam(E)$ for all $i\in \{0,\dots,n\}$. If $F=\{p_{w_0},\dots,p_{w_n}\}$, then there exist $i,j\in \{0,\dots,n\}$ such that
$$\diam(F)\leq \diam(U_{w_i})+\diam(U_{w_j}) +\diam(E)\leq C_4(L)\diam(E).$$
This completes the proof.
\end{proof}

We say that a metric space $X$ is \textit{quasiconvex} if there exists $L\geq 1$ such that for every pair of points $x,y\in X$ there exists a curve $\gamma$ connecting $x$ and $y$ with $\ell(\gamma)\leq Ld_X(x,y)$. In this case we say that $X$ is $L$-quasiconvex.

\begin{lemma}\label{lemma:quasiconvex}
Let $X$ be a metric space, $K,L\geq 1$, and $\mathcal A=((V,\sim),\mathfrak{p},\mathfrak{r},\mathcal U)$ be a $(K,L)$-approximation of $X$. Suppose that $X$ is $L$-quasiconvex. Then for every pair of vertices $u,v\in V$ with $k(u,v)\geq K$ there exist vertices $u=w_0,w_1,\dots,w_n=v$, $n\geq 0$, with $w_{i-1}\sim w_i$ for each $i\in \{1,\dots,n\}$, such that
\begin{align}\label{lemma:quasiconvex:main}
\sum_{i=0}^n r_{w_i} \leq C(K,L) d_X(p_u,p_v).
\end{align}
\end{lemma}

\begin{proof}
Let $u,v\in V$ be distinct vertices. Let $\gamma$ be a curve connecting $p_u$ and $p_v$ with $\ell(\gamma)\leq L d_X(p_u,p_v)$. Note that $U_u\cap U_v=\emptyset $ by \ref{a:3}, so by \ref{a:2} we have 
\begin{align}\label{lemma:quasiconvex:lower}
r_u+r_v\leq \ell(\gamma)\leq Ld_X(p_u,p_v).
\end{align}
Suppose, first, that $|\gamma|$ is contained in $N(U_w,r_w/L)$ for some $w\in V$. Then $U_u\cap N(U_w,r_w/L)\neq \emptyset$ so by \ref{a:6} we conclude that $k(u,w)<2K$ and similarly $k(v,w)<2K$, so $k(u,v)<4K$. We consider a chain $u=w_0,\dots,w_n=v$ of length $n=k(u,v)$ with $w_{i-1}\sim w_i$ for $i\in \{1,\dots,n\}$. By \ref{a:3} we have $r_{w_i}\leq C(K,L) r_u$ for each $i\in \{0,\dots,n\}$. Thus, in combination with \eqref{lemma:quasiconvex:lower} we have
$$\sum_{i=0}^n r_{w_i} \leq (n+1)C(K,L)r_u \leq (4K+1)C(K,L) \ell(\gamma).$$
This completes the proof in that case.

Next, suppose that $|\gamma|$ is not contained in $N(U_w,r_w/L)$ for any $w\in V$. The collection $\{U_w: w\in V,\, U_w\cap |\gamma|\neq \emptyset\}$ is an open cover of $|\gamma|$. The connectedness of $|\gamma|$ implies that there exist vertices $u=w_0,w_1,\dots,w_n=v$, where $n\geq 0$, such that $U_{w_i}\cap |\gamma|\neq \emptyset$ for $i\in \{0,\dots,n\}$ and  $U_{w_{i-1}}\cap U_{w_i}\neq \emptyset$ for $i\in \{1,\dots,n\}$. By \ref{a:3}, we have $k(w_{i-1},w_i)<K$ for $i\in \{1,\dots,n\}$. Since $|\gamma|$ is not contained in $N(U_{w_i},r_{w_i}/L)$, by \ref{a:4} we have
\begin{align}\label{lemma:quasiconvex:integral}
\int_\gamma \chi_{\st_K(w_i)} \, ds \geq  \int_{\gamma} \chi_{N(U_{w_i},r_{w_i}/L)}\, ds \geq L^{-1}r_{w_i}
\end{align}
for each $i\in \{0,\dots,n\}$.

We claim that for each $x\in X$ we have
\begin{align}\label{lemma:quasiconvex:multiplicity}
\sum_{w\in V} \chi_{\st_K(w)}(x) \leq C(K).
\end{align}
We fix $w\in V$. Suppose that $x\in \st_K(w)\cap \st_K(w')$ for some $w'\in V$, $x\in X$. Then there exist $s,s'\in V$ such that $k(w,s)<K$, $x\in U_s$, $k(w',s')<K$, and $x\in U_{s'}$. Since $U_s\cap U_{s'}\neq \emptyset$, by \ref{a:3} we have $k(s,s')<K$, so $k(w,w')<3K$. By \ref{a:1} we conclude that there are at most $C(K)$ vertices $w'$ such that $\st_K(w')$ intersects $\st_K(w)$. This proves the claim. We now integrate both sides of \eqref{lemma:quasiconvex:multiplicity} and by \eqref{lemma:quasiconvex:integral} we obtain
\begin{align*}
C(K)\ell(\gamma) \geq \sum_{i=0}^n \int_{\gamma} \chi_{\st_K(w_i)}\, ds \geq L^{-1} \sum_{i=0}^n r_{w_i}.
\end{align*}

It only remains to satisfy the requirement that $w_{i-1}\sim w_i$. This can be easily achieved by adding, for each $i\in \{1,\dots,n\}$, at most $K$ vertices in a combinatorial path from $w_{i-1}$ to $w_i$. For each added vertex $w$ the value of $r_w$ is comparable to $r_{w_i}$, thanks to \ref{a:3}. This completes the proof.
\end{proof}

\subsection{Images of approximations of metric spaces}
In the statements below, such as in Lemma \ref{lemma:approximation_qs_local}, if $Y$ is a metric space, then the metric of $Y$ is implicitly understood to be $d_Y$, whenever this is not explicitly stated. In some cases, such as in Lemma \ref{lemma:approximation_change}, where we consider different metrics on the same underlying set $Y$, we will explicitly state the metric that we are using each time.

\begin{definition}[Image of approximation under a homeomorphism]\label{definition:image_approximation}
Let $X,Y$ be connected metric spaces, $f\colon X\to Y$ be a homeomorphism, and $K,L\geq 1$. Suppose that $\mathcal A=((V,\sim),\mathfrak{p},\mathfrak{r}, \mathcal U)$ is a fine $(K,L)$-approximation of $X$. Define $\mathcal A'= ((V,\sim), \mathfrak{p}',\mathfrak{r}', \mathcal U')$, where for $v\in V$ we set $\mathfrak p'(v)={p}'_v= f(p_v)$, $\mathcal U'(v)=U_v'=f(U_v)$, and
$$r_v'=\inf\{d_Y(f(x),p_v'): x\in X, d_X(x,p_v)\geq r_v\}.$$
We say that $\mathcal A'$ is the \textit{image of the approximation $\mathcal A$ of $X$ under $f$} and we write $\mathcal A'=f(\mathcal A)$. 
\end{definition}
Since $\mathcal A$ is a fine approximation of $X$, we have  $U_v\neq X$ for each $v\in V$; this and the continuity of $f^{-1}$ imply that $r_v'\in (0,\infty)$ for each $v\in V$. If $f$ is $\eta$-quasisymmetric, then by \cite{BonkKleiner:quasisphere}*{Lemma 4.1}, $\mathcal A'$ is a $(K,L')$-approximation of $Y$ for some $L'$ depending only on $K,L$ and $\eta$. In general, without any assumptions on $f$ or $Y$,  $\mathcal A'$ is not an approximation of $Y$. In the next lemma we provide another sufficient condition for $\mathcal A'$ to be an approximation of $Y$.

\begin{lemma}[Image of approximation under a quasisymmetry]\label{lemma:approximation_qs_local}
Let $X$, $Y$ be connected metric spaces, $f\colon X\to Y$ be a homeomorphism, $\eta\colon [0,\infty)\to [0,\infty)$ be a homeomorphism, and $K,L\geq 1$. Suppose that $\mathcal A=((V,\sim),\mathfrak{p},\mathfrak{r}, \mathcal U)$ is a fine $(K,L)$-approximation of $X$ and the following conditions hold.
\begin{enumerate}[label=\normalfont{(\arabic*)}]
	\item $Y$ has $L$-bounded turning.
	\item  For each $v\in V$, $f|_{\st_{K}(v)}$ is $\eta$-quasisymmetric.
\end{enumerate}
Then $\mathcal A'=f(\mathcal A)$ is a fine $(K,L')$-approximation of $(Y,d_Y)$, where $L'=C(K,L,\eta)$. 
\end{lemma}
The assumption that $Y$ has bounded turning compensates for the fact that $f$ is not quasisymmetric on all of $X$.

\begin{proof}
We use the notation $x'=f(x)$ for $x\in X$ and $x=f^{-1}(x')$ for $x'\in Y$. Note that \ref{a:1} is automatically true for $\mathcal A'$.  We first show that for each $v\in V$,
\begin{align}\label{lemma:approximation_qs_local_outside_star}
\text{if $x\notin \st_K(v)$ and $y\in U_v$, then $d_Y(x',y')\gtrsim_{K,L,\eta} \diam(U_v')$.}
\end{align}
Let $x,y$ as above. Since $Y$ has $L$-bounded turning, there exists a connected set $E'$ connecting $x'$ to $y'$ such that $\diam(E')\leq L d_Y(x',y')$. Let $E=f^{-1}(E')$. By \ref{a:4},  $N(U_v,{r_v/L})\subset \st_K(v)$, so the connectedness of $E$ implies that  the set $F= E\cap \st_K(v)$ satisfies $\diam(F)\geq r_v/L$. In particular, by \ref{a:2} and \ref{a:3}, we have $\diam(F)\simeq_{K,L} \diam(\st_K(v))\simeq_{K,L} r_v$. Since $f$ is a quasisymmetry from $\st_K(v)$ onto $f(\st_K(v))$, we conclude that $\diam(f(F))\simeq_{K,L,\eta} \diam(f(\st_K(v)))$; see Lemma \ref{lemma:qs_heinonen}. Thus, $$d_Y(x',y')\gtrsim_L \diam(E') \gtrsim_L \diam(f(F)) \simeq_{K,L,\eta} \diam(f(\st_K(v))) \gtrsim_{K,L,\eta} \diam(U_v').$$
This completes the proof of the claim.

We now establish \ref{a:2} for $\mathcal A'$. Let $v\in V$. By definition, we have $B(p_v', r_v')\subset f(B(p_v,r_v))\subset f(U_v)=U_v'$, so $\diam(U_v')\geq r_v'$. Recall that $B(p_v,r_v)\neq X$ since $\mathcal A$ is a fine approximation. By the connectedness of $X$, there exists $x_v\in \br{B(p_v,r_v)}$ with $d_X(p_v,x_v)=r_v$. Suppose that $x\in X$ and $d_X(p_v,x)\geq r_v$, as in the definition of $r_v'$. By \ref{a:2} and \ref{a:3} we have $\diam(\st_K(v))\simeq_{K,L} r_v$. If $x\in \st_{K}(v)$, then the fact that $d_X(p_v,x)\gtrsim_{K,L} \diam(\st_K(v))$ and the assumption that $f|_{\st_{K}(v)}$ is quasisymmetric, with the aid of Lemma \ref{lemma:qs_heinonen}, yield 
$$d_Y(p_v',x') \gtrsim_{K,L,\eta} \diam(f(\st_K(v)))\gtrsim_{K,L,\eta} \diam(U_v').$$
If $x\notin \st_K(v)$, then by \eqref{lemma:approximation_qs_local_outside_star} we have $d_Y(p_v',x')\gtrsim_{K,L,\eta} \diam(U_v')$. If we combine both cases, the definition of $r_v'$ yields 
\begin{align}\label{lemma:approximation_qs_local_a3}
\diam(U_v')\simeq_{K,L,\eta}r_v'.
\end{align}
This proves \ref{a:2}.
  
Next, we show \ref{a:3} for $\mathcal A'$. If $u\sim v$, then $U_v\cap U_u\neq \emptyset$ and $r_u\simeq_L r_v$ by property \ref{a:3} for $\mathcal A$. Combined with \ref{a:2}, this gives $\diam(U_u)\simeq_L \diam(U_v)$. Note that $U_u\cup U_v\subset \st_K(v)$ by \ref{a:3}. Since $f$ quasisymmetric on $\st_K(v)$, by Lemma \ref{lemma:qs_heinonen} we conclude that $\diam(U_u')\simeq_{L,\eta} \diam(U_v')$. Combining this with \eqref{lemma:approximation_qs_local_a3}, we obtain $r_u'\simeq_{K,L,\eta} r_v'$. This completes the proof of the first part of \ref{a:3}. The rest of \ref{a:3} is combinatorial so it holds trivially.

Finally \ref{a:4} follows immediately from \eqref{lemma:approximation_qs_local_outside_star} and \eqref{lemma:approximation_qs_local_a3}.
\end{proof}

\begin{lemma}[Changing the metric on vertices]\label{lemma:approximation_change}
Let $(Y,d_Y)$ be a metric space, $K,L,M\geq 1$, and $\mathcal A=((V,\sim),\mathfrak{p},\mathfrak{r}, \mathcal U)$ be a fine $(K,L)$-approximation of $Y$. Let $S=\mathfrak{p}(V)$ and $d_S$ be a metric on $S$, and suppose that the following conditions hold.
\begin{enumerate}[label=\normalfont{(\arabic*)}]
	\item\label{lemma:approximation_bt:1} $(Y,d_Y)$ has $L$-bounded turning.
	\item\label{lemma:approximation_change:ineq} $d_S\leq d_Y$ on $S\times S$.
	\item\label{lemma:approximation_change:lower} For every pair of distinct vertices $u,v\in V$ we have $d_S(p_u,p_v)\geq r_v/L$.
	\item\label{lemma:approximation_bt:2} For every pair of vertices $u,v\in V$ there exist vertices $u=w_0, w_1, \dots, w_n=v$, $n\geq 0$, with $k(w_{i-1},w_i)<K$ for each $i\in \{1,\dots,n\}$, such that 
$$\diam_{d_S} (\{p_{w_0},\dots,p_{w_n} \})\leq L d_S(p_u,p_v).$$
\end{enumerate} 
Consider the metric $\widetilde d$ on $Y$ arising by gluing $d_Y$ with $d_S$, as in Lemma \ref{lemma:glue}. Then the following statements are true.
\begin{enumerate}[label=\normalfont{(\roman*)}]
	\item\label{lemma:approximation_change:bilip} Let $u,v\in V$ and $x\in U_u$, $y\in U_v$. Then 
	$$\min \{r_u+r_v, d_Y(x,y)\}\leq C(K,L)\widetilde d(x,y).$$
	In particular, if $k(u,v)<M$, then
	$$ \widetilde d(x,y)\leq d_Y(x,y)\leq C(K,L,M) \widetilde d(x,y).$$
	\item\label{lemma:approximation_change:bt} $(Y,\widetilde d)$ has $C(K,L)$-bounded turning. 
	\item\label{lemma:approximation_change_image} Let $\widetilde {\mathcal A}$ be the image of $\mathcal A$ under the identity map from $(Y,d_Y)$ onto $(Y,\widetilde d)$.  Then $\widetilde {\mathcal A}$ is a fine $(K,\widetilde L)$-approximation of $(Y,\widetilde d)$, where $\widetilde L=C(K,L)$.
\end{enumerate}
\end{lemma}
\begin{proof}
First, we show \ref{lemma:approximation_change:bilip}.  Let $x,y\in X$ and $u,v\in V$ such that $x\in U_u$ and $y\in U_v$. Let $s,t\in V$. If $s=t$, then
$$d_Y(x,p_s)+d_S(p_s,p_t)+d_Y(p_t,y)\geq d_Y(x,y).$$
Suppose that $s\neq t$. If $k(u,s)<2K$ and $k(v,t)<2K$, then by  assumption \ref{lemma:approximation_change:lower} and \ref{a:3}, we have
$$d_Y(x,p_s)+d_S(p_s,p_t)+d_Y(p_t,y)\geq d_S(p_s,p_t)\geq L^{-1}\max\{r_s,r_t\} \gtrsim_{K,L}(r_u+r_v).$$
If $k(u,s)<2K$ and $k(v,t)\geq 2K$, then by  assumption \ref{lemma:approximation_change:lower}, \ref{a:3}, and \ref{a:6}, we have
\begin{align*}
d_Y(x,p_s)+d_S(p_s,p_t)+d_Y(p_t,y)\geq L^{-1}r_s+ L^{-1}r_v \gtrsim_{K,L}(r_u+r_v).
\end{align*}
The same is true when $k(u,s)\geq 2K$ and $k(v,t)<2K$. Finally, suppose that $k(u,s)\geq 2K$ and $k(v,t)\geq 2K$. Then by \ref{a:6} we have
$$d_Y(x,p_s)+d_S(p_s,p_t)+d_Y(p_t,y)\geq L^{-1}(r_u+r_v).$$
Combining all cases, we see that 
$$\inf_{s,t\in V}\{d_Y(x,p_s)+d_S(p_s,p_t)+d_Y(p_t,y) \} \gtrsim_{K,L} \min \{r_u+r_v, d_Y(x,y)\}.$$
The definition of $\widetilde d$ gives the first part of \ref{lemma:approximation_change:bilip}. For the second part, we note that if $k(u,v)<M$, then by \ref{a:2} and \ref{a:3} we obtain $d_Y(x,y)\lesssim_{L,M} r_u+r_v$. 

Next, we show \ref{lemma:approximation_change:bt}. Since $\widetilde d\leq d_Y$, the identity map from $(Y,d_Y)$ onto $(Y,\widetilde d)$ is continuous and if a set is connected in $(Y,d_Y)$, then it is also connected in $(Y,\widetilde d)$.  Let $x,y\in X$ and $u,v\in V$ such that $x\in U_u$ and $y\in U_v$. We consider two cases.

\smallskip
\noindent
\textit{Case 1:}  $k(u,v)<2K$. Then by the fact that $\widetilde d\leq d_Y$, assumption \ref{lemma:approximation_bt:1}, and conclusion \ref{lemma:approximation_change:bilip}, there exists a connected set $E$ containing $x,y$ with 
$$\diam_{\widetilde d}(E)\leq \diam_{d_Y}(E) \leq Ld_Y(x,y) \lesssim_{K,L}\widetilde d(x,y).$$

\smallskip
\noindent
\textit{Case 2:} $k(u,v)\geq 2K$. In this case, by \ref{a:6} we have $d_Y(x,y)\geq L^{-1}\max\{r_v,r_u\}$. Combining this with conclusion \ref{lemma:approximation_change:bilip} and with \ref{a:2}, we see that 
\begin{align}\label{lemma:approximation_change:case2}
\begin{aligned}
\widetilde d(x,y)&\gtrsim_{K,L} (r_u+r_v) \gtrsim_{K,L} (d_Y(p_u,x)+d_Y(p_v,y))\\
&\gtrsim_{K,L} (\widetilde d(p_u,x)+\widetilde d(p_v,y)).
\end{aligned}
\end{align}
Also, by Case 1, there exist a connected set $E_x$ containing $x$ and $p_u$ and a connected set $E_y$ containing $y$ and $p_v$ with 
\begin{align}\label{lemma:approximation_change:case2a}
\diam_{\widetilde d}(E_x)\lesssim_{K,L}\widetilde d(x,p_u)\quad\text{and}\quad \diam_{\widetilde d}(E_y) \lesssim_{K,L}\widetilde d(y,p_v). 
\end{align}
By assumption \ref{lemma:approximation_bt:2}, there exist vertices $u=w_0, w_1, \dots, w_n=v$, $n\geq 0$, with $k(w_{i-1},w_i)<K$ for each $i\in \{1,\dots,n\}$, such that
\begin{align}\label{lemma:approximation_change:case2b}
\diam_{d_S} (\{p_{w_0},\dots,p_{w_n} \})\leq Ld_S(p_u,p_v).
\end{align}
By Case 1, for each $i\in \{1,\dots,n\}$ there exists a connected set $E_i$ containing $p_{w_{i-1}}$ and $p_{w_i}$ such that $\diam_{\widetilde d}(E_i)\lesssim_{K,L} \widetilde d(p_{w_{i-1}},p_{w_i})$.   Let $E=\bigcup_{i=1}^n E_i$, which is connected. Then there exist $i,j\in \{0,\dots,n\}$ such that 
\begin{align*}
\diam_{\widetilde d}(E) &\leq \diam_{\widetilde d}(E_i)+  \widetilde d(p_{w_{i}},p_{w_j})+ \diam_{\widetilde d}(E_j)\\
&\lesssim_{K,L} \widetilde d(p_{w_{i-1}},p_{w_i}) + \widetilde d(p_{w_{i}},p_{w_j})+ \widetilde d(p_{w_{j-1}},p_{w_j})\\
&\simeq_{K,L} d_S(p_{w_{i-1}},p_{w_i}) +  d_S(p_{w_{i}},p_{w_j})+ d_S(p_{w_{j-1}},p_{w_j})\\
&\lesssim_{K,L} d_S(p_u,p_v)\simeq_{K,L}\widetilde d(p_u,p_v),
\end{align*}
where we used the fact that $d_S=\widetilde d$ on $S\times S$ (see Lemma \ref{lemma:glue}) and \eqref{lemma:approximation_change:case2b}. The set $E\cup E_x\cup E_y$ is connected and contains $x$ and $y$. Combining the above with \eqref{lemma:approximation_change:case2a} and \eqref{lemma:approximation_change:case2},  we have
\begin{align*}
\diam_{\widetilde d}(E\cup E_x\cup E_y) &\lesssim_{K,L} \widetilde d(p_u,p_v) + \widetilde d(p_u,x) +\widetilde d(p_v,y)\\
&\lesssim_{K,L} \widetilde d(p_u,x)+ \widetilde d(x,y)+ \widetilde d(p_v,y)\\
&\lesssim_{K,L} \widetilde d(x,y).
\end{align*}
This completes the proof of \ref{lemma:approximation_change:bt}.

For \ref{lemma:approximation_change_image}, consider the identity map from $(Y,d_Y)$ onto $(Y,\widetilde d)$. The space $(Y,\widetilde d)$ has $C(K,L)$-bounded turning by \ref{lemma:approximation_change:bt}. Also, by \ref{lemma:approximation_change:bilip}, for each $v\in V$, the identity map from $(\st_K(v),d_Y)$ onto $(\st_K(v), \widetilde d)$ is $C(K,L)$-bi-Lipschitz, and thus $\eta$-quasisymmetric for some $\eta$ depending only on $K$ and $L$. Therefore, the assumptions of Lemma \ref{lemma:approximation_qs_local} are satisfied for the identity map and we conclude that $\widetilde {\mathcal A}$ is a $(K,\widetilde L)$-approximation of $(Y,\widetilde d)$, where $\widetilde L=C(K,L)$. The fineness of $\widetilde {\mathcal A}$ follows from the fineness of $\mathcal A$.
\end{proof}

\subsection{Approximations and quasisymmetric homeomorphisms}
We state the main result of the section. Essentially it allows us to upgrade a local quasisymmetry between metric spaces $X$, $Y$ to a global one and even provides the flexibility of changing the metric in large scales of the target space $Y$.

\begin{theorem}\label{theorem:approximation_quasisymmetric}
Let $(X,d_X),(Y,d_Y)$ be connected metric spaces, $f\colon X\to Y$ be a homeomorphism, $\eta\colon [0,\infty)\to [0,\infty)$ be a homeomorphism, and $K,L\geq 1$. Let $\mathcal A=((V,\sim),\mathfrak{p},\mathfrak{r}, \mathcal U)$ be a fine {$(K,L)$-approximation} of $(X,d_X)$ and suppose that the following conditions are satisfied.
\begin{enumerate}[label=\normalfont{(\arabic*)}] 
	\item\label{tqs:bt} $(X,d_X)$ and $(Y,d_Y)$ have $L$-bounded turning.
	\item\label{tqs:sep} For every pair of distinct vertices $u,v\in V$ we have $d_X(p_u,p_v)\geq r_u/L$.
	\item\label{tqs:star}  For each $v\in V$,  $f|_{\st_{2K+1}(v)}\colon (\st_{2K+1}(v),d_X)\to (f(\st_{2K+1}(v)),d_Y)$ is $\eta$-quasi\-symmetric.
	\item\label{tqs:vert} For $S=f(\mathfrak p(V))$ there exists a metric $d_S\colon S\times S\to[0,\infty)$ such that $f\colon (\mathfrak{p}(V), d_X)\to (S,d_S)$ is $\eta$-quasisymmetric and $d_S\leq d_Y$ on $S\times S$. Moreover, if $u,v\in V$ and $u\sim v$, then $d_S(f(p_u),f(p_v))\geq L^{-1}d_Y(f(p_u),f(p_v))$.
\end{enumerate}
Consider the metric $\widetilde d$ on $Y$ arising by gluing $d_Y$ with $d_S$, as in Lemma \ref{lemma:glue}. Then there exists a homeomorphism $\eta'\colon [0,\infty)\to [0,\infty)$ depending only on $K,L$, and $\eta$, such that $f\colon (X,d_X)\to (Y,\widetilde d)$ is $\eta'$-quasisymmetric. 
\end{theorem}

We establish a preliminary statement before giving the proof of Theorem \ref{theorem:approximation_quasisymmetric}. We believe that the constant $2K+1$ appearing in assumption \ref{tqs:star} above and also in the lemma below can be replaced with $K$, but we do not include a proof for the sake of brevity.

\begin{lemma}[Local to global quasisymmetry]\label{lemma:wqs}
Let $X,Y$ be connected metric spaces, $f\colon X\to Y$ be a homeomorphism, $\eta\colon [0,\infty)\to [0,\infty)$ be a homeomorphism, and $K,L\geq 1$. Suppose that $\mathcal A=((V,\sim),\mathfrak{p},\mathfrak{r}, \mathcal U)$ is a fine $(K,L)$-approximation of $X$ and the following conditions hold. \begin{enumerate}[label=\normalfont{(\arabic*)}]
	\item\label{wqs:approx} $\mathcal A'=f(\mathcal A)$ is a $(K,L)$-approximation of $Y$.
	\item\label{wqs:star} For each $v\in V$,  $f|_{\st_{2K+1}(v)}$ is $\eta$-quasisymmetric. 
	\item\label{wqs:vertices} $f|_{\mathfrak p(V)}$ is $\eta$-quasisymmetric. 
\end{enumerate}
Then there exists a homeomorphism $\eta'\colon [0,\infty)\to [0,\infty)$ depending only on $K,L$, and $\eta$ such that $f\colon X\to Y$ is $\eta'$-quasisymmetric. 
\end{lemma}

\begin{proof}
For $x\in X$, we use the notation $x'=f(x)$. Let $x,y,z\in X$ such that $d_X(x,y)=t d_X(x,z)$ for some $t>0$. Our goal is to show that the images $x',y',z'$ satisfy $$d_Y(x',y')\leq \eta'(t)\cdot  d_Y(x',z')$$
for some distortion function $\eta'$ that depends only on $K,L$, and $\eta$. Consider vertices $u,v,w\in V$ such that $x\in U_{u}$, $y\in U_v$, and $z\in U_w$. We consider several cases for the relative positions of these vertices. 

\smallskip
\noindent
\textit{Case 1:} $k(u,v)<2K$ and $k (u,w)<2K$. Then $x,y,z\in \st_{2K}(u)$. By assumption \ref{wqs:star}, $f$ is an $\eta$-quasisymmetry on $\st_{2K+1}(u)$, so $d_Y(x',y')\leq \eta(t) d_Y(x',z')$.

\smallskip
\noindent
\textit{Case 2:} $k(u,v)\geq 2K$ and $k(u,w)\geq 2K$. By \ref{a:7}, we have
\begin{align*}
d_{X}(p_u,p_v) \simeq_L d_X(x,y)\quad \text{and} \quad d_Y(p_u',p_v') \simeq_L d_Y(x',y'),
\end{align*}
and
\begin{align*}
d_{X}(p_u,p_w) \simeq_L d_X(x,z)\quad \text{and} \quad d_Y(p_u',p_w') \simeq_L d_Y(x',z')
\end{align*}
Thus, we have $d_X(p_u,p_v)\leq C(L) \cdot t \cdot d_X(p_u,p_w)$. Since $f|_{\mathfrak p(V)}$ is $\eta$-quasisymmetric, we conclude that $d_Y(p_u',p_v') \leq \eta(C(L)\cdot t) d_Y(p_u',p_w')$, from which we obtain 
$$d_Y(x',y')\leq C(L)\eta(C(L)\cdot t) d_Y(x',z').$$

\smallskip
\noindent
\textit{Case 3:} $k(u,v)\geq 2K$ and  $k(u,w)<2K$. Let $s\in V$ be such that $2K\leq  k(u,s) <2K+1$; such $s$ exists because $k(u,v)\geq 2K$ and the graph $(V,\sim)$ is connected by \ref{a:5con}. Property \ref{a:7} implies that $d_X(p_u,p_v)\simeq_L d_X(x,y)$ and the latter equals $t\cdot d_X(z,z)$ by assumption. Since $x,z\in \st_{2K}(u)$, we have $d_X(x,z)\lesssim_{K,L}r_u$, by \ref{a:2} and \ref{a:3}. Finally, by \ref{a:6} we have $r_u\lesssim_L d_X(p_u,p_s)$. Combining these estimates, we have
$$d_X(p_u,p_v)\lesssim_{K,L}t \cdot  d_X(p_u,p_s).$$
Since $f$ is $\eta$-quasisymmetric on $\mathfrak{p}(V)$, we obtain
\begin{align*}\
d_Y(p_u',p_v') \leq \eta(C(K,L)\cdot t) d_Y(p_u',p_s').
\end{align*}
Invoking again \ref{a:7}, we obtain $d_Y(p_u',p_v')\simeq_L d_Y(x',y')$, so
\begin{align}\label{lemma:wqs:3p}
d_Y(x',y')\leq C(L) \eta(C(K,L)\cdot t)d_Y(p_u',p_s').
\end{align}
It remains to bound the term $d_Y(p_u',p_s')$ in the right-hand side by a certain multiple of $d_Y(x',z')$.

By the condition $k(u,s)<2K+1$, \ref{a:2}, and \ref{a:3}, we have $d_X(x,p_s)\lesssim_{K,L}r_u$. By \ref{a:3}, we have $U_{u}\cap U_{v}=\emptyset$, so $r_{u}\leq d_X(p_u,p_v)$. Therefore, by the above,
$$d_X(x,p_s)\lesssim_{K,L}r_u\lesssim_{K,L}d_X(p_u,p_v)\simeq_{K,L} d_X(x,y)\simeq_{K,L}t\cdot d_X(x,z).$$
Note that $x,p_s,z\in \st_{2K+1}(u)$. Since $f$ is $\eta$-quasisymmetric in $\st_{2K+1}(u)$, we conclude that
\begin{align}\label{lemma:wqs:3y}
d_Y(x',p_s')\leq \eta(C(K,L)\cdot t)d_Y(x',z').
\end{align}
By \ref{a:2} and \ref{a:3}, $d_Y(p_u',p_s')\lesssim_{K,L}r_u'$. By \ref{a:6} we have $r_u'\lesssim_L d_Y(x',p_s')$. These two estimates and \eqref{lemma:wqs:3y} give
$$d_Y(p_u',p_s')\leq C(K,L) \eta(C(K,L) \cdot t) d_Y(x',z').$$
This estimate, combined with \eqref{lemma:wqs:3p} gives the desired
$$d_Y(x',y')\leq C(K,L) \eta(C(K,L)\cdot t)^2 d_Y(x',z').$$

\smallskip
\noindent
\textit{Case 4:}  $k(u,v)<2K$ and $k(u,w)\geq 2K$. As in the previous case, let $s\in V$ be such that $2K\leq  k(u,s) <2K+1$. Consider $t_1>0$ such that $d_X(x,y)=t_1 d_X(x,p_s)$. By \ref{a:2} and \ref{a:3}, we have $d_X(x,y)\lesssim_{K,L}r_u$. Also, $d_X(x,p_s)\gtrsim_{L}r_u$ by \ref{a:6}. These estimates give 
$$t_1\lesssim_{K,L}1.$$
Note that $x,y,p_s\in \st_{2K+1}(u)$ and since $f$ is an $\eta$-quasisymmetry on $\st_{2K+1}(u)$, we have 
\begin{align}\label{lemma:wqs:4t1}
d_Y(x',y')\leq \eta(t_1) d_Y(x',p_s').
\end{align}

Next, we have $t_1 d_X(x,p_s)=d_X(x,y)=td_X(x,z)$, so $d_X(x,p_s)=t_2d_X(x,z)$, where $t_2=t/t_1$. We have $d_X(x,p_s)\lesssim_{K,L}r_u$ by \ref{a:2} and \ref{a:3}. Also, $d_X(x,z) \gtrsim_L r_u$ by \ref{a:6}. Combining these estimates, we have $$t_2\lesssim_{K,L}1.$$
Since $k(u,s)\geq 2K$, $k(u,w)\geq 2K$, and $d_X(x,p_s)=t_2d_X(x,z)$, by applying Case 2, we have
$$d_Y(x',p_s')\leq C(L)\eta(C(L)\cdot t_2)d_Y(x',z').$$
Combining this inequality with \eqref{lemma:wqs:4t1}, we have
$$d_Y(x',y') \leq C(L)\eta(t_1) \cdot \eta(C(L)\cdot t_2) d_Y(x',z').$$
If $t_1\leq t_2$, then, given that $t_2\lesssim_{K,L}1$, we have
$$\eta(t_1) \cdot \eta(C(L)\cdot t_2)\leq \eta(\sqrt{t_1t_2}) \cdot \eta(C(K,L))=C(K,L,\eta)\cdot \eta(\sqrt{t}).$$
If $t_2<t_1$, then we similarly have
$$\eta(t_1) \cdot \eta(C(L)\cdot t_2) \leq C(K,L,\eta)\cdot \eta(C(L)\sqrt{t}).$$
Summarizing, 
$$d_Y(x',y') \leq C(K,L,\eta)\eta(C(L)\sqrt{t})d_Y(x',z').$$
This completes the proof of Case 4.

\medskip

Finally, if we combine Cases 1--4, we obtain that $f$ is an $\eta'$-quasisymmetry for
$$\eta'(t)=C(K,L,\eta) \max \{ \eta(C(L)\cdot t), \eta(C(K,L)\cdot t)^2, \eta(C(L)\sqrt{t})\}.$$
This completes the proof. 
\end{proof}

\begin{proof}[Proof of Theorem \ref{theorem:approximation_quasisymmetric}]
Our goal is to show that $f\colon (X,d_X)\to (Y,\widetilde d)$ satisfies the assumptions of Lemma \ref{lemma:wqs}. Note that assumption \ref{wqs:vertices}  of Lemma \ref{lemma:wqs} follows immediately from assumption \ref{tqs:vert} of Theorem \ref{theorem:approximation_quasisymmetric}; recall that $\widetilde d=d_S$ on $S\times S$ by Lemma \ref{lemma:glue} \ref{g:cases}.  It remains to verify assumptions \ref{wqs:approx} and \ref{wqs:star} of Lemma \ref{lemma:wqs}. 

We first ensure that Lemma \ref{lemma:approximation_change} is applicable. By assumptions \ref{tqs:bt} and \ref{tqs:star} of Theorem \ref{theorem:approximation_quasisymmetric}, and Lemma \ref{lemma:approximation_qs_local}, $\mathcal A'=f(\mathcal A)$ is a fine $(K,L')$-ap\-prox\-i\-mation of $(Y,d_Y)$ with $L'=C(K,L,\eta)$, as required in Lemma \ref{lemma:approximation_change}. Note that assumptions \ref{lemma:approximation_bt:1} and \ref{lemma:approximation_change:ineq} of Lemma \ref{lemma:approximation_change} are true, directly from the assumptions of Theorem \ref{theorem:approximation_quasisymmetric}. 

Next, we establish that condition \ref{lemma:approximation_change:lower} in Lemma \ref{lemma:approximation_change} holds.  Let $u,v\in V$ with $u\sim v$. By assumption \ref{tqs:sep} of Theorem \ref{theorem:approximation_quasisymmetric} and \ref{a:2}--\ref{a:3}, we have $d_X(p_u,p_v)\simeq_{K,L}\diam_X(\st_K(v))$. Assumption \ref{tqs:star} of Theorem \ref{theorem:approximation_quasisymmetric} and Lemma \ref{lemma:qs_heinonen} imply that 
$$d_Y(p_u',p_v') \simeq_{K,L,\eta} \diam_{d_Y}(f(\st_K(v))).$$
Since $\mathcal A'$ is a $(K,L')$-ap\-prox\-i\-mation of $(Y,d_Y)$, we have $\diam_{d_Y}(f(\st_K(v)))\simeq_{K,L'}r_v'$ by \ref{a:2}--\ref{a:3}. Altogether, 
$$d_Y(p_u',p_v') \simeq_{K,L,\eta} r_v'.$$
By assumption \ref{tqs:vert} of Theorem \ref{theorem:approximation_quasisymmetric}, $d_S(p_u',p_v')\simeq_L d_Y(p_u',p_v')$. Hence, 
$$d_S(p_u',p_v')\simeq_{K,L,\eta}r_v'.$$
This shows condition \ref{lemma:approximation_change:lower} of Lemma \ref{lemma:approximation_change} in the case that $u\sim v$. 

Now, suppose that $u,v\in V$ are arbitrary and distinct. Let $w\in V$ with $w\sim u$. Then $d_X(p_u,p_w)\simeq_{K,L} r_u$ and $d_S(p_u',p_w')\simeq_{K,L,\eta}r_u'$ by the above. By assumption \ref{tqs:sep} of Theorem \ref{theorem:approximation_quasisymmetric}, we have $d_X(p_u,p_v)\gtrsim_{K,L} d_X(p_u,p_w)$. By assumption \ref{tqs:vert} of Theorem \ref{theorem:approximation_quasisymmetric} we have $d_S(p_u',p_v')\gtrsim_{K,L,\eta} d_S( p_u',p_w')\simeq_{K,L,\eta} r_u'$. This completes the proof of condition \ref{lemma:approximation_change:lower} of Lemma \ref{lemma:approximation_change} in that case as well.

We proceed with assumption \ref{lemma:approximation_bt:2} of Lemma \ref{lemma:approximation_change}. The assumption that $X$ has $L$-bounded turning implies that the conclusion of Lemma \ref{lemma:approximation_connected} is true on the set $\mathfrak{p}(V)$. By Lemma \ref{lemma:qs_heinonen}, this conclusion is invariant up to constants under quasisymmetries from $(\mathfrak{p}(V),d_X)$ to $(S,d_S)$. Hence, assumption \ref{tqs:vert} of Theorem \ref{theorem:approximation_quasisymmetric} implies that the conclusion of Lemma \ref{lemma:approximation_connected} holds for the set $S$ with metric $d_S$, quantitatively. Equivalently, condition \ref{lemma:approximation_bt:2} in Lemma \ref{lemma:approximation_change} holds with a constant $C(L,\eta)$ in place of $L$. 

We have established all assumptions of  Lemma \ref{lemma:approximation_change}. By Lemma \ref{lemma:approximation_change} \ref{lemma:approximation_change_image} we conclude that the image $\widetilde {\mathcal A}$ of $\mathcal A'$ under the identity map from $(Y,d_Y)$ onto $(Y,\widetilde d)$ is a $(K,\widetilde L)$-approximation of $(Y,\widetilde d)$, where $\widetilde L= C(K,L,\eta)$. This verifies condition \ref{wqs:approx} of Lemma \ref{lemma:wqs}. By Lemma \ref{lemma:approximation_change} \ref{lemma:approximation_change:bilip}, identity map from $(f(\st_{2K+1}(v)), d_Y)$ onto $(f(\st_{2K+1}(v)), \widetilde d)$ is $C(K,L,\eta)$-bi-Lipschitz for each $v\in V$. By assumption \ref{tqs:star} of  Theorem \ref{theorem:approximation_quasisymmetric}, we conclude that $f\colon (\st_{2K+1}(v), d_X)\to (f(\st_{2K+1}(v)), \widetilde d)$ is quasisymmetric, quantitatively, as required in Lemma \ref{lemma:wqs} \ref{wqs:star}. We have verified all assumptions of Lemma \ref{lemma:wqs}, as desired.
\end{proof}

\section{Simplicial complexes}\label{section:simplicial}

In this section we introduce metric and quasiconformal simplicial complexes. Then we discuss approximations of those complexes in the sense of Section \ref{section:approximation}.

\subsection{Metric simplicial complexes}

We provide some definitions. We direct the reader to \cite{BridsonHaefliger:metric}*{Chapter I.7} for further details.

Let $n,m\in \N\cup \{0\}$ with $n\leq m$. An \textit{$n$-simplex} $S\subset \R^m$ is the convex hull of $n+1$ points in general position, i.e., not lying in the same $n$-dimensional plane. The points are called the \textit{vertices} of $S$. If the mutual distance of vertices is $1$, then $S$ is called the \textit{standard $n$-simplex} (which is unique up to isometry). A \textit{face} $T\subset S$ is the convex hull of a non-empty subset of vertices of $S$. Each face $T$ of $S$ is a $k$-simplex for some $k\leq n$. 

\begin{definition}
Let $\{S_{\lambda}\}_{\lambda\in \Lambda}$ be a family of simplices. Let $Z=\bigcup_{\lambda\in \Lambda} (S_\lambda\times \{\lambda\})$ and $\sim$ be an equivalence relation on $Z$. The space $X=Z/\sim $ is a \textit{simplicial complex} if the natural projection $p\colon Z\to X$ satisfies the following conditions.
\begin{enumerate}[label=\normalfont{(\arabic*)}]
	\item For every $\lambda\in \Lambda$ the map $p_\lambda=p(\cdot,\lambda)\colon S_{\lambda}\to X$ is injective.
	\item If $p_{\lambda}(S_\lambda)\cap p_{\lambda'}(S_{\lambda'})\neq \emptyset$, then for each point $q\in p_{\lambda}(S_\lambda)\cap p_{\lambda'}(S_{\lambda'})$ there exists a face $T_\lambda$ of $S_\lambda$, a face $T_{\lambda'}$ of $S_{\lambda'}$, and an isometry $h_{\lambda,\lambda'}\colon T_\lambda\to T_{\lambda'}$ such that $q\in p_{\lambda}(T_\lambda)\cap p_{\lambda'}(T_{\lambda'})$ and
	$$\text{$p(x,\lambda)=p(x',\lambda')$ if and only if $x'=h_{\lambda,\lambda'}(x)$.}$$
\end{enumerate}
\end{definition}

By the above definition, the vertices and faces of $S_\lambda$, $\lambda\in \Lambda$, yield naturally the notions of vertices and faces of $X$. As a consequence, if a face of $X$ is contained in a union of faces, then it must be contained in one of them. Let $S\subset p_{\lambda}(S_{\lambda})$ be a face of $X$. We define $d_S(p_\lambda(x),p_{\lambda}(y))$ to be the Euclidean distance of the points $x$ and $y$. This gives a length metric on $S$.

Let $x,y\in X$. A \textit{string} from $x$ to $y$ is a finite sequence $\Sigma=(z_0,z_1,\dots,z_m)$, where $x=z_0,z_m=y$, $m\in \N$, such that $z_{i-1},z_i$ lie in the same face $S(i)\subset X$ for each $i\in \{1,\dots,m\}$. The length of $\Sigma$ is defined as 
\begin{align*}
\ell(\Sigma)= \sum_{i=1}^m d_{S(i)}(z_{i-1},z_i).
\end{align*} 
The \textit{trace} of $\Sigma$ is the union of the line segments from $z_{i-1}$ to $z_i$, $i\in \{1,\dots,m\}$, and is denoted by $|\Sigma|$. The \textit{intrinsic pseudometric} on $X$ is defined by
\begin{align*}
d_X(x,y)\coloneqq \inf\{ \ell(\Sigma) : \text{$\Sigma$ is a string from $x$ to $y$} \}.
\end{align*}
If there is no such string, then $d_X(x,y)=\infty$. If $d_X$ is a metric, i.e., for every pair of distinct points $x,y$ we have $0<d_X(x,y)<\infty$, then $(X,d_X)$ is called a \textit{metric simplicial complex}. In that case, $(X,d_X)$ is a length space.

Let $X, Y$ be simplicial complexes. A map $f\colon X\to Y$ is \textit{simplicial} if it maps each simplex of $X$ linearly onto a simplex of $Y$. We say that $f$ is a simplicial isomorphism if $f$ is in addition bijective. In that case, $f^{-1}$ is also a simplicial map.

Let $X$ be a simplicial complex. For a point $x\in X$, let $S(x)$ be the collection of all simplices containing $x$. For $S\in S(x)$, define $\varepsilon (x,S)$ to be the distance in the length metric of $S$ from $x$ to the union of faces of $S$ that do not contain $x$. Also, let $\varepsilon(x)=\inf\varepsilon (x,S)$ where the infimum is over all simplices $S\in S(x)$. Then 
\begin{align}\label{simplicial:basic}
\begin{aligned}
&B_{d_X}(x,\varepsilon(x))\subset \bigcup S(x) \quad \text{and} \\
&d_{X}(x,y)=d_S(x,y) \,\,\, \text{for}\,\,\, y\in B_{d_X}(x,\varepsilon(x))\cap S,\,\,\, \text{where}\,\,\, S\in S(x). 
\end{aligned}
\end{align}
See \cite{BridsonHaefliger:metric}*{Lemma I.7.9} for an argument. We list some elementary properties of simplicial complexes.

\begin{enumerate}[label=\normalfont{(SC\arabic*)}]
	\item\label{sc1} If $\varepsilon(x)>0$ for each $x\in X$ then $(X,d_X)$ is a length space \cite{BridsonHaefliger:metric}*{Corollary I.7.10}.
	\item\label{sc2} Let $x\in X$. If $S(x)$ contains finitely many simplices, then $\varepsilon(x)>0$. In that case, we say that $X$ is \textit{locally finite} at $x$. If $X$ is locally finite at each point, then we say that $X$ is \textit{locally finite}. By \ref{sc1}, a locally finite simplicial complex is a length space.
	\item\label{sc3} Let $f$ be a {simplicial map} from $X$ onto a simplicial complex $Y$. If $X$ is locally finite, it is a consequence of \eqref{simplicial:basic} that $f$ is continuous. If $f$ is a simplicial isomorphism, then $f$ is a homeomorphism. 
\end{enumerate}

\subsection{Quasiconformal simplicial complexes}
We introduce the notion of a quasiconformal simplicial complex, by requiring that the simplices have uniform geometry.
\begin{definition}\label{definition:qc_complex}
Let $X$ be a simplicial complex and $M\geq 1$. We say that $X$ is an $M$-\textit{quasiconformal simplicial complex} if the following conditions hold.
\begin{enumerate}[label=\normalfont{(QC-SC\arabic*)},wide=\parindent, leftmargin=*]
	\item\label{qcsc:1} For each $x\in X$, $\#S(x)\leq M$. That is, $x$ belongs to at most $M$ simplices.
	\item\label{qcsc:2} For each $n$-simplex $S\subset X$, where $n\in \N$, there exists a linear map $\tau$ from $S$ onto the standard $n$-simplex in $\R^n$ such that
	$$M^{-1} \frac{d_S(x,y)}{\diam_{d_S}(S)}\leq  |\tau(x)-\tau(y)|\leq M \frac{d_S(x,y)}{\diam_{d_S}(S)}$$
	for all $x,y\in S$. 
	\item\label{qcsc:3} For each $x\in X$, if $S,T\in S(x)$ are non-degenerate simplices, then
	$$M^{-1}\leq  \frac{\diam_{d_S}(S)}{\diam_{d_T}(T)}\leq M.$$
\end{enumerate}
\end{definition}

\begin{remark}
The bound $\#S(x)\leq M$ implies that $X$ does not have any $n$-simplices for $n>M$.  
\end{remark}

Our goal is to prove that a simplicial isomorphism between quasiconformal simplicial complexes is locally bi-Lipschitz in a quantitative sense; see Lemma \ref{lemma:simplicial:approximation:isomorphism}.

\begin{lemma}\label{lemma:simplicial:lower}
Let  $X$ be a simplicial complex and $x,y\in X$ be distinct vertices. Then $d_X(x,y)\geq \varepsilon(x)$. If, in addition, there exists $M\geq 1$ such that $X$ is an $M$-quasiconformal simplicial complex and the vertices $x,y$  lie in a simplex $S\subset X$, then $$\varepsilon(x)\leq d_X(x,y)\leq \diam_{d_X}(S)\leq \diam_{d_S}(S)\leq  c(M)\varepsilon(x).$$
\end{lemma}
\begin{proof}
Let $x,y\in X$ be distinct vertices. We show the first inequality. If $d_X(x,y)<\varepsilon(x)$, then, by \eqref{simplicial:basic}, $y\in B_{d_X}(x,\varepsilon(x))\subset \bigcup S(x)$, so $y\in S$ for some $S\in S(x)$, and $d_X(x,y)=d_S(x,y)$. By the definition of $\varepsilon(x,S)$, we have $d_S(x,y)\geq \varepsilon(x,S)\geq \varepsilon(x)$. This is a contradiction.

Next, observe that the height of the standard $n$-simplex is $\sqrt{(n+1)/(2n)}\geq \sqrt{2}/2$, $n\in \N$. Thus the distance of each vertex of the standard simplex to the faces that do not contain it is at least $\sqrt{2}/2$. This implies that if $X$ is $M$-quasiconformal and $x$ is a vertex, then by condition \ref{qcsc:2}, for each $S\in S(x)$ we have 
$$\varepsilon(x,S)\geq M^{-1}\frac{\sqrt{2}}{2}\diam_{d_S}(S).$$
If $T\in S(x)$ is any other simplex, then by the above and \ref{qcsc:3} we have
$$\varepsilon (x,T) \geq M^{-2}\frac{\sqrt{2}}{2}\diam_{d_S}(S).$$
Thus, $\varepsilon (x)\geq c(M) \diam_{d_S}(S)$. The remaining inequalities in the statement of the lemma are trivial.
\end{proof}

Let $X$ be a simplicial complex and $A\subset X$ be a set. Define $S(A)$ to be the collection of all simplices of $X$ that intersect $A$. 

\begin{prop}\label{proposition:simplicial:distance}
Let $M\geq 1$ and $X$ be an $M$-quasiconformal simplicial complex.  If $A$ is an $n$-simplex of $X$, where $n\in \N$, then for $r=c(M)\diam_{d_S}(A)$, we have
$$N_{d_X}( A,r) \subset \inter\left(\bigcup S(A)\right).$$
\end{prop}

\begin{proof}
We will show by induction that if $A$ is an $n$-simplex in $X$, $n\in \N$, then 
$$N_{d_X}(A,r)\subset \bigcup S(A)$$
where $r=C(M) \min\{\varepsilon (x): \text{$x$ is a vertex of $A$}\}$ and $C(M)=M^{-2}\sqrt{2}/4$. By Lemma \ref{lemma:simplicial:lower}, $r$ is comparable to $\diam_{d_S}(A)$. The desired conclusion follows upon observing that $N_{d_X}(A,r)$ is an open set. 

For $n=1$, consider an edge $A=[x_1,x_2]$ of $X$.  Let  $c_1=\min \{\varepsilon (x_1),\varepsilon(x_2)\}$ and observe that $B_{d_X}(x_i,c_1)\subset \bigcup S(x_i)$ for $i=1,2$ by \eqref{simplicial:basic}.  Let $x\in A$ and consider two cases.
First, suppose that $x\in B_{d_X}(x_i,c_1/2)$ for some $i=1,2$. Then 
\begin{align}\label{proposition:simplicial:distance:1}
B_{d_X}(x,c_1/2)\subset B_{d_X}(x_i,c_1)\subset \bigcup S(x_i)\subset \bigcup S(A).
\end{align}
Next, suppose that $x\in A\setminus (B_{d_X}(x_1,c_1/2)\cup B_{d_X}(x_2,c_1/2))$. Let $S\in S(x)$ be a $k$-simplex of $X$, $k\in \N$. If $S=A$, the length within $A$ from $x$ to each endpoint of $A$ is at least $c_1/2$, so $\varepsilon (x, S)\geq c_1/2$. Next, suppose that $S\neq A$.  We map $S$ to the standard $k$-simplex with a map $\tau$ satisfying condition \ref{qcsc:2}. We will estimate $\varepsilon(\tau(x),\tau(S))$ from below. Let $T$ be a face of $\tau(S)$ such that $\varepsilon(\tau(x),\tau(S))=\dist(\tau(x),T)$. A simple calculation shows that $T$ has a common vertex $z$ with the edge $\tau(A)$. 
We consider a scaling $\phi(y)= z+R\cdot (y-z)$ so that the segment $[z,\tau(x)]$ is an edge of the simplex  $\phi(\tau(S))$; see Figure \ref{figure:edge}. The edge length of that simplex is at least $ M^{-1} (c_1/2) \diam_{d_S}(S)^{-1}$. The distance $\dist(\tau(x),T)$ is the height of $\phi(\tau(S))$ at the vertex $\tau(x)$. Given that the height of the standard simplex is bounded below by $\sqrt{2}/2$, we obtain
$$\varepsilon(\tau(x),\tau(S)) \geq (\sqrt{2}/2) M^{-1} (c_1/2) \diam_{d_S}(S)^{-1}.$$
Therefore, $\varepsilon(x,S)\geq  M^{-2} (\sqrt{2}/4)c_1\eqqcolon c_2$. We conclude that $\varepsilon(x)\geq \min\{c_1/2,c_2\}=c_2$. Therefore, by \eqref{simplicial:basic},
\begin{align}\label{proposition:simplicial:distance:2}
B_{d_X}(x,c_2)\subset B_{d_X}(x,\varepsilon(x))\subset \bigcup S(x)\subset \bigcup S(A).
\end{align}
Combining \eqref{proposition:simplicial:distance:1} and \eqref{proposition:simplicial:distance:2}, and given that $c_2\leq c_1/2$, we obtain 
$$N_{d_X}(A,c_2)\subset \bigcup S(A).$$
This completes the proof in the case that $n=1$. 

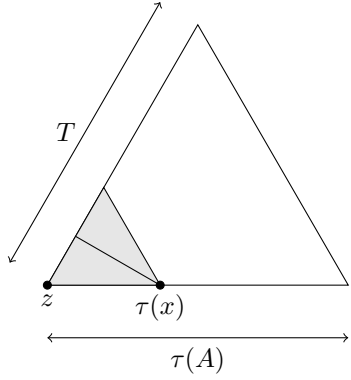
\begin{figure}
\centering
\begin{tikzpicture}
	\draw (0,0)-- (4,0) -- (60:4)--(0,0);
	\coordinate (x) at (1.5,0);
	\coordinate (z) at (0,0);
	\draw[yshift=.3cm,xshift=-.5cm,<->] (0,0)--(60:4) node[pos=0.5,left]{$T$};
	\draw[yshift=-.7cm,<->] (0,0)--(4,0) node[pos=0.5, below]{$\tau(A)$};
	\draw[fill=black!10] (z)--(x)--(60:1.5)--cycle;
	\draw[fill=black] (x) circle (1.5pt) node[below] {$\tau(x)$};
	\draw[fill=black] (z) circle (1.5pt) node[below] {$z$};
	\draw (x)--(60:0.75);
\end{tikzpicture}
\caption{Shown is the simplex $\tau(S)$ and the shaded simplex represents $\phi(\tau(S))$.}\label{figure:edge}
\end{figure}

For the general case we argue by induction. Suppose that the statement is true for $n$-simplices, $n\in \N$. Let $A$ be an $(n+1)$-simplex and let $T_i$, $i\in \{0,\dots,n+1\}$, be $n$-simplices that are faces of $A$. By the induction assumption, we have
$$N_{d_X}(T_i,r_i) \subset \bigcup S(T_i),$$
where $r_i=C(M) \min\{\varepsilon (x): \text{$x$ is a vertex of $T_i$}\}$. Then observe that $S(A)=\bigcup_{i=0}^{n+1} S(T_i)$ and $r_i\geq r\coloneqq C(M) \min\{\varepsilon (x): \text{$x$ is a vertex of $A$}\}$ for $i\in \{0,\dots,n+1\}$, so
$$N_{d_X}(A,r) \subset A\cup \bigcup_{i=0}^{n+1}N_{d_X}(T_i,r_i) \subset \bigcup S(A).$$
This completes the proof. 
\end{proof}

\begin{lemma}\label{lemma:simplicial:approximation}
Let $M\geq 3$ and $X$ be an $M$-quasiconformal simplicial complex. Let $(V,\sim)$ be a graph that corresponds to the $1$-skeleton of $X$ under an embedding $\mathfrak{p}\colon V\to X$, and define $\mathfrak{r}(v)=\varepsilon (\mathfrak{p}(v))$ and $\mathcal U(v)= \inter(\bigcup S(\mathfrak{p}(v)))$ for $v\in V$. Then $\mathcal A=((V,\sim), \mathfrak{p},\mathfrak{r},\mathcal U)$ is an $(M,M')$-approximation of $X$, where $M'=C(M)$.
\end{lemma}

Recall the definition of an approximation and the relevant notation from Section \ref{section:approximation}.

\begin{proof}
Property \ref{a:1} follows from condition \ref{qcsc:1} in the definition of a quasiconformal simplicial complex. For \ref{a:2} we have 
$$B_{d_X}(p_v,\varepsilon (p_v)) \subset U_v$$
by \eqref{simplicial:basic}. Also, by Lemma \ref{lemma:simplicial:lower} we have $\diam_{d_X}(U_v)\simeq_M \varepsilon(p_v)$. This completes the proof of \ref{a:2}. 

For \ref{a:3}, if $u\sim v$, then the interior of an edge connecting $p_u$ and $p_v$ lies in $U_u\cap U_v$. Moreover, $r_u\simeq_M r_v$ by Lemma \ref{lemma:simplicial:lower}. Conversely, if $U_u\cap U_v\neq \emptyset$, then there exists a simplex $S$ of $X$ whose interior is contained in both $U_u$ and $U_v$. We conclude that $S\in S(p_u)$ and $S\in S(p_v)$.  Thus, $p_u$ and $p_v$ are vertices of  $S$ and $k(u,v)\leq 1<M$.

Finally, we show \ref{a:4}. Note that if $S$ is a simplex of $X$ with vertices $u_0,\dots,u_n$, then 
$$S \subset \bigcup_{i=0}^n U_{u_i}.$$
Let $v\in V$ and $A\in S(p_v)$ be an $n$-simplex, $n\in \N$. Then $\diam_{d_S}(A) \simeq_M \varepsilon (p_v)$ by Lemma \ref{lemma:simplicial:lower}. By Proposition \ref{proposition:simplicial:distance} and the above, we have
\begin{align*}
N_{d_X}(A, r)&\subset \bigcup S(A)\subset  \bigcup \{U_u: \text{$u$ is a vertex of some $S\in S(A)$}\}\\
&\subset \st_{3}(v)=\bigcup\{U_u:u\in V,\, k(u,v)<3\}
\end{align*}
for $r=C(M)\diam_{d_S}(A)\simeq_M \varepsilon (p_v) \simeq_M r_v$. Thus,
$$N_{d_X}(U_v, r)\subset \st_{3}(v)$$
for some $r\simeq_M r_v$. Since $M\geq 3$, this completes the proof of \ref{a:4}.
\end{proof}

\begin{lemma}\label{lemma:simplicial:approximation:isomorphism}
Let $L\geq 1$, $M\geq 3$, and $X,Y$ be $M$-quasiconformal simplicial complexes. Let $f\colon X\to Y$ be a simplicial isomorphism. Let $\mathcal A=((V,\sim),\mathfrak{p},\mathfrak{r},\mathcal U)$ and $\mathcal A'=((V,\sim),\mathfrak{p}',\mathfrak{r}',\mathcal U')$ be the $(M,M')$-approximations of $X$ and $Y$, respectively, that are provided by Lemma \ref{lemma:simplicial:approximation}.  Then for each $v\in V$ and $x,y\in \mathcal A$-$\st_L(v)$ we have
\begin{align*}
C(L,M)^{-1}\frac{d_X(x,y)}{r_v}\leq \frac{d_Y(f(x),f(y))}{r_v'}\leq C(L,M) \frac{d_X(x,y)}{r_v}.
\end{align*}
In particular, $f|_{\mathcal A\text{-}\st_L(v)}$ is $\eta$-quasisymmetric, where $\eta(t)=C(L,M)^2t$.
\end{lemma}

\begin{proof}
We use the notation $x'=f(x)$ for points $x\in X$ and analogous notation for images of simplices and strings. Let $S\subset X$ be an $n$-simplex, $n\in \N$, and let $S'=f(S)$. We claim that 
\begin{align}\label{lemma:simplicial:approximation:simplex}
M^{-2}\frac{d_S(x,y)}{\diam_{d_S}(S)} \leq \frac{d_{S'}(x',y')}{\diam_{d_{S'}}(S')}\leq M^{2}\frac{d_S(x,y)}{\diam_{d_S}(S)}. 
\end{align}
To see this, note that by \ref{qcsc:2} there exist linear maps $\tau,\tau'$ from $S,S'$, respectively, onto the standard $n$-simplex in $\R^n$ satisfying the inequality stated in \ref{qcsc:2}.  The composition $\tau' \circ f\circ \tau^{-1}$ is a linear automorphism of the standard $n$-simplex, so it is an isometry. This suffices for \eqref{lemma:simplicial:approximation:simplex}.

To prove the inequalities in the statement of the lemma, it suffices to show the right inequality and then apply it to $f^{-1}$ to obtain the left one. Suppose that $x,y\in \Ast_L(v)$ for some $v\in V$ and $x\neq y$. Then there exists $u\in V$ with $k(u,v)<L$ such that  $x\in U_u$. By \ref{a:3} we have 
\begin{align}\label{lemma:simplicial:approximation:isomorphism:radius}
r_u\simeq_{L,M} r_v
\end{align}
Let $\Sigma=(z_0,z_1,\dots,z_m)$ be a string in $X$ connecting $x$ and $y$, where $z_0=x$ and $z_m=y$. Also, since $x\neq y$, we can assume that there exist non-degenerate simplices $S(i)\subset X$, $i\in \{1,\dots,m\}$, such that $z_{i-1},z_i\in S(i)$ for $i\in \{1,\dots,m\}$. We consider two cases.

\smallskip
\noindent
\textit{Case 1:} $|\Sigma|$ is not contained in $\Ast_M(u)$. Since $|\Sigma|$ intersects $U_u$, by \ref{a:4} we have $\ell(\Sigma)\gtrsim_M r_u$.  On the other hand, $x',y'\in \A'st_L(v)$, so $d_Y(x',y') \lesssim_{L,M} r_v'$ by \ref{a:2} and \ref{a:3}. Altogether, combining the above with \eqref{lemma:simplicial:approximation:isomorphism:radius}, we obtain 
$$\frac{d_Y(f(x),f(y))}{r_v'}\lesssim_{L,M} 1\lesssim_{L,M}\frac{\ell(\Sigma)}{r_v}.$$

\smallskip
\noindent
\textit{Case 2:} $|\Sigma|$ is contained in $\Ast_M(u)$. Note that $\Ast_M(u)$ is contained in the union of simplices one vertex of which has combinatorial distance less than $M$ from $u$. Thus, for each $i\in \{1,\dots,m\}$ the simplex $S(i)$ intersects one of those simplices. In other words, there exists $v_i\in V$ with $k(u,v_i)<M$ such that $S(i)$ intersects a simplex $T(i)$ that has $v_i$ as a vertex. By Lemma \ref{lemma:simplicial:lower}, $r_{v_i}=\varepsilon(p_{v_i})\simeq_M \diam_{d_{T(i)}}(T(i))$. By \ref{qcsc:3}, we have $\diam_{d_{T(i)}}(T(i))\simeq_M \diam_{d_{S(i)}}(S(i))$. Altogether, $\diam_{d_{S(i)}}(S(i))\simeq_M r_{v_i}$. Also, $r_u\simeq_{M} r_{v_i}$ by \ref{a:3} and $r_v\simeq_{L,M} r_{v_i}$ by \eqref{lemma:simplicial:approximation:isomorphism:radius}. The analogous statements are true for the images under $f$. By the above and  \eqref{lemma:simplicial:approximation:simplex}, we have
\begin{align*}
\frac{d_Y(x',y')}{r_v'} &\leq \frac{\ell(\Sigma')}{r_v'} \simeq_{L,M}
\sum_{i=1}^m \frac{d_{S'(i)} (z_{i-1}',z_i')}{\diam_{d_{S'(i)}}(S'(i))}\\
&\simeq_{L,M} \sum_{i=1}^m \frac{d_{S(i)} (z_{i-1},z_i)}{\diam_{d_{S(i)}}(S(i))}
\simeq_{L,M} \frac{\ell(\Sigma)}{r_v}.
\end{align*}
Combining both cases and infimizing over all strings $\Sigma$ connecting $x$ and $y$, we obtain the conclusion.
\end{proof}

\subsection{Constructions}
We include an elementary construction of a quasiconformal simplicial complex whose boundary is a triangle that has prescribed side lengths.

\begin{lemma}\label{lemma:triangle}
Let $d_1,d_2,d_3$ be positive numbers such that $d_i\leq d_j+d_k$ for all distinct $i,j,k\in \{1,2,3\}$. Let $M\geq 1$ such that
$$M^{-1}\leq \frac{d_i}{d_j}\leq M$$
for each $i,j\in \{1,2,3\}$. Let $S$ be the standard $2$-simplex with edges $e_1,e_2,e_3$. Then there exists a $C(M)$-quasiconformal simplicial complex $K$ consisting of three $2$-simplices and a simplicial complex $S'$ that arises by connecting the barycenter of $S$ to the vertices with the following properties: there exists a simplicial isomorphism $f\colon S'\to K$ that maps the edge $e_i$ of $S'$ to an edge of $K$ of length $d_i$ for $i\in \{1,2,3\}$.
\end{lemma}

\begin{proof}
By scaling, we may assume that $d_1=1$, so we have $M^{-1}\leq d_i\leq M$ for $i\in \{1,2,3\}$. Consider a possibly degenerate triangle $T$ in the plane with sides $K_1,K_2,K_3$ of length $d_1,d_2,d_3$, respectively. The distance of the barycenter of $T$ to each vertex or side is bounded above by $2M/3\leq M$. Consider a point $z_0\in \R^2\times \{1\}$ that projects to the barycenter of $T$. The distance of $z_0$ to each vertex is bounded above by $\sqrt{1+ 4M^2/9}\leq 2M$ and below by $1$.  Denote by $\widetilde K_i$ the convex hull of $K_i$ and $z_0$. Then $\widetilde K_i$ is a triangle of side lengths in the interval $[M^{-1},2M]$. Also, the height of the triangle $\widetilde K_i$ from $z_0$ to the side $K_i$ is bounded below by $1$. If $\theta$ is an angle between two sides of $\widetilde K_i$, then 
$$\frac{1}{2M}\leq \frac{1}{2}d_i\cdot 1\leq \area(\widetilde K_i) \leq \frac{1}{2} (2M)^2\sin \theta.$$ 
This implies that each angle of $\widetilde K_i$ is bounded from below away from $0$, depending on $M$. Thus, there exists a linear and $C(M)$-bi-Lipschitz map from $\widetilde K_i$ onto the standard $2$-simplex.

We now consider a simplicial complex $K$ arising by considering the disjoint union $\widetilde K_1\sqcup \widetilde K_2\sqcup\widetilde K_3$ and identifying, for each pair of distinct $i,j\in \{1,2,3\}$, the edge of $\widetilde K_i$ with the edge of $\widetilde K_j$ that connect the common vertex of $K_i$ and $K_j$ to $z_0$. Note that if $T$ is a non-degenerate triangle, then $K$ is realized as a subset of $\R^3$ obtained by taking the union of $\widetilde K_1,\widetilde K_2,\widetilde K_3$, as subsets of $\R^3$; see Figure \ref{figure:lemma:triangle}. By the above, $K$ is a $C(M)$-quasiconformal simplicial complex. We give consistent orientations to the triangles $\widetilde K_i$, $i\in \{1,2,3\}$.

\begin{figure}
\begin{tikzpicture}
\begin{scope}[shift={(6,0)}]
\coordinate (A) at (4,0,0);
\coordinate (B) at (2,0,1.5);
\coordinate (C) at (2,4,0.3);
\coordinate (O) at (0,0,0);

\fill (O) circle (1.5pt);
\fill (A) circle (1.5pt);
\fill (B) circle (1.5pt);

\draw[dashed] (O) --(A);
\draw  (O) -- (B) -- (A);
\draw[color=black, fill=black!20, fill opacity=0.5] (O)-- (B) -- (C) -- cycle;
\draw (A)--(B)--(C)--cycle;

\fill (C) circle (1.5pt) node[anchor=west] {$z_0$};
\node[anchor=north east] at ($0.5*(O)+0.5*(B)$) {$K_i$};
\node at ($0.33*(O)+0.33*(B)+0.33*(C)$) {$\widetilde K_i$};
\end{scope}

\draw[->] (3.5,2) to[out=30,in=150](5.5,2);
\node at (4.5,2.6) {$f$};

\begin{scope}[shift={(-6,0)}]
\coordinate (a) at (6,0);
\coordinate (b) at (9,0);
\coordinate (c) at (7.5, {3*sqrt(3)*0.5});
\coordinate (d) at ($0.333*(a)+0.333*(b)+0.333*(c)$);

\draw (a) -- (b) -- (c) -- cycle;
\draw(d)--(c);
\draw[color=black, fill=black!10] (a)node[xshift=1.5cm,below]{$e_i$}--(b) --(d)--cycle;
\fill (a) circle (1.5pt);
\fill (b) circle (1.5pt);
\fill (c) circle (1.5pt);
\fill (d) circle (1.5pt);
\end{scope}
\end{tikzpicture}
\caption{The simplicial complex $K$ constructed in Lemma \ref{lemma:triangle}.}\label{figure:lemma:triangle}
\end{figure}
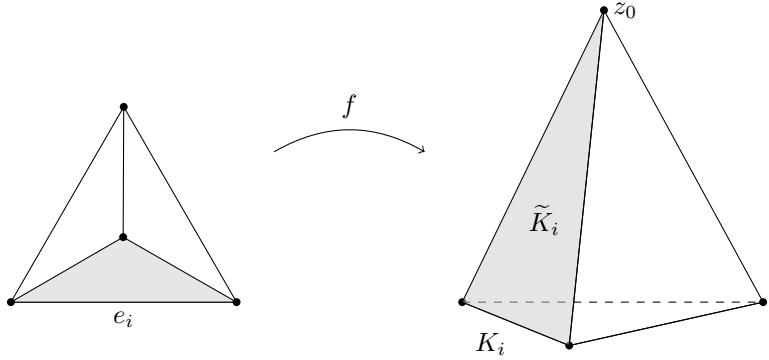

Consider an equilateral triangle $S$ of unit side length with edges $e_1,e_2,e_3$. We subdivide $S$ into three triangles $S_1,S_2,S_3$ by connecting the barycenter to the vertices, so that $e_i\subset S_i$. Thus, we obtain a simplicial complex $S'$ consisting of these three triangles. For $i\in \{1,2,3\}$, consider an orientation-preserving linear map from $S_i$ onto $\widetilde K_i$ so that the edge $e_i$ of $S_i$ is mapped to the edge $K_i$ of $\widetilde K_i$. 
Pasting together these three maps gives rise to the desired simplicial isomorphism from $S'$ onto $K$. 
\end{proof}

\section{Simplicial and smooth approximation of manifolds}\label{section:proof}

This section is devoted to the proofs of Theorem \ref{theorem:main:approximation} and Theorem \ref{theorem:main:approximation:bilip}.

\subsection{Gromov--Hausdorff convergence}
Let $(X,d_X)$ be a metric space and let $E\subset X$ and $\varepsilon>0$. We say that $E$ is \textit{$\varepsilon$-dense} (in $X$) if for each $x\in X$ we have $\dist_{d_X}(x,E)<\varepsilon$ or equivalently $N_{d_X}(E,\varepsilon)=X$. A map $f \colon X \to Y$ (not necessarily continuous) between metric spaces is an \textit{$\varepsilon$-isometry} if $f(X)$ is $\varepsilon$-dense in $Y$ and 
$$|d_X(x,y) - d_Y(f(x),f(y))| < \varepsilon$$
for each $x,y \in X$.

We define the \textit{Hausdorff distance} of two sets $E,F\subset X$ to be the {infimal value} $r>0$ such that $E\subset N_{d_X}(F,r)$ and $F\subset N_{d_X}(E,r)$. We denote the Hausdorff distance by $d_H(E,F)$. A sequence of sets $E_n\subset X$ \textit{converges in the Hausdorff sense} to a set $E\subset X$ if $d_H(E_n,E)\to 0$ as $n\to\infty$.

The \textit{Gromov--Hausdorff distance} between two metric spaces $X,Y$ is defined as the infimal value $r>0$ such that there is a metric space $Z$ with subsets $\widetilde{X}, \widetilde{Y} \subset Z$ such that $X$ and $Y$ are isometric to $\widetilde{X}$ and $\widetilde{Y}$, respectively, and $d_H(\widetilde{X}, \widetilde{Y}) < r$. This is denoted by $d_{GH}(X,Y)$. We note that if there exists an $\varepsilon$-isometry from $X$ to $Y$ for some $\varepsilon>0$, then $d_{GH}(X,Y)<2\varepsilon$ \cite{BuragoBuragoIvanov:metric}*{Corollary 7.3.28}. We say that a sequence of metric spaces $X_n$, $n\in \N$, \textit{converges in the Gromov--Hausdorff sense} to a metric space $X$ if $d_{GH}(X_n,X) \to 0$ as $n \to \infty$. By \cite{BuragoBuragoIvanov:metric}*{Corollary 7.3.28}, this is equivalent to the requirement that there exists a sequence of $\varepsilon_n$-isometries $f_n\colon X_n\to X$, where $\varepsilon_n>0$ and $\varepsilon_n\to 0$ as $n\to\infty$. In this case, we say that $f_n$, $n\in \N$, is an \textit{approximately isometric sequence}. See \cite{BuragoBuragoIvanov:metric}*{Section 7} and  \cite{Petrunin:metric_geometry}*{Section 5} for more background.

\begin{lemma}[\cite{KeithLaakso:Assouaddimension}*{Lemma 2.4.7}]\label{lemma:keithlaakso}
Let $\{(X_n,d_n)\}_{n\in \N}$ and $\{(Y_n,\rho_n)\}_{n\in \N}$ be sequences of compact metric spaces that converge in the Gromov--Hausdorff sense to compact metric spaces $(X,d)$ and $(Y,\rho)$, respectively. Let $f_n\colon X_n\to Y_n$, $n\in \N$, be an $\eta$-quasisymmetric homeomorphism, where $\eta$ is a fixed distortion function. Furthermore, suppose that there exist points $a_n,b_n\in X_n$, $n\in \N$, and $C\geq 1$ such that
\begin{align*}
C^{-1}\leq d_n(a_n,b_n)\leq C\quad \text{and}\quad C^{-1}\leq \rho_n(f_n(a_n),f_n(b_n))\leq C
\end{align*}
for all $n\in \N$. Then the sequence $\{f_n\}_{n\in \N}$ has a subsequence that converges to an $\eta$-quasisymmetric homeomorphism $f\colon X\to Y$.
\end{lemma}
The convergence of $\{f_n\}_{n\in \N}$ has to be interpreted appropriately, but we do not go into detail for the sake of brevity.

\begin{proof}[Proof of Theorem \ref{theorem:main:approximation} \ref{theorem:main:approximation:2} $\Rightarrow$ \ref{theorem:main:approximation:1}]
Suppose that \ref{theorem:main:approximation:2} is true. Consider an approximately isometric sequence $f_k\colon (X,d_k)\to (X,d_X)$, $k\in \N$, and a sequence of $\eta$-quasisymmetries $h_k\colon (X,d_g)\to (X,d_k)$, $k\in \N$. By the definition of an approximately isometric sequence, for sufficiently large $k\in\N$ we have 
$$\frac{1}{2}\leq \frac{\diam_{d_k}(X)}{\diam_{d_X}(X)}\leq \frac{3}{2}.$$  
We fix points $a,b\in X$. By Lemma \ref{lemma:qs_heinonen} we have
\begin{align*}
\frac{d_k(h_k(a),h_k(b))}{\diam_{d_k}(X)} \geq \frac{1}{2\eta\left( \frac{\diam d_g(X)}{d_g(a,b)}\right)}.
\end{align*}
Therefore, for sufficiently large $k\in \N$, we have
\begin{align*}
\frac{\diam_{d_X}(X)}{4\eta\left( \frac{\diam d_g(X)}{d_g(a,b)}\right)}\leq d_k(h_k(a),h_k(b)) \leq \frac{3}{2}\diam_{d_X}(X).
\end{align*}
By Lemma \ref{lemma:keithlaakso}, there exists an $\eta$-quasisymmetric homeomorphism $h\colon (X,d_g)\to (X,d_X)$.
\end{proof}

\subsection{Simplicial approximations and smoothing}
We will use two recent results of Bowditch and Cattalani regarding simplicial approximation of Riemannian manifolds and smoothing of simplicial complexes, respectively. 

\begin{theorem}[\cite{Bowditch:triangulation}*{Theorems 1.2 and 1.3}]\label{theorem:bowditch}
Let $X$ be a compact Riemannian $n$-manifold with boundary. Then there exists $\eta_0>0$ such that for each $t\in (0,\eta_0)$ there exists a locally finite simplicial complex $Z$ whose $n$-simplices are copies of the standard $n$-simplex scaled by $t$, and there exists a $c(n)$-bi-Lipschitz homeomorphism $\tau\colon Z \to X$. 
\end{theorem}

See also the relevant work of Boissonnat--Dyer--Ghosh \cite{BoissonnatDyerGhosh:triangulation}. 

\begin{remark}\label{remark:degree}
The degree of each vertex in the $1$-skeleton of $Z$ is bounded above, depending only on the dimension $n$. Indeed, if a number $m$ of $n$-simplices of $Z$ meet at a vertex $x$ of $Z$, then the ball $B(x,r)$ for small $r$ has volume $C(n)m r^n$. Since $\tau$ is $c(n)$-bi-Lipschitz, the image of the ball $B(x,r)$ under $\tau$ has volume comparable to $m r^n$  and is contained in a ball of radius comparable to $r$ in the Riemannian metric of $X$. If $r$ is small enough, then we see that $m$ has to be uniformly bounded.
\end{remark}

\begin{theorem}[\cite{Cattalani:smoothing}*{Theorem 2}]\label{theorem:cattalani}
Let $M\geq 1$ and $X$ be an $M$-quasiconformal simplicial complex that is homeomorphic to a $2$-manifold. Then there exists a $C(M)$-bi-Lipschitz homeomorphism from $X$ onto a Riemannian manifold. 
\end{theorem}

\subsection{Proof of Theorem \ref{theorem:main:approximation} \ref{theorem:main:approximation:1} $\Rightarrow$ \ref{theorem:main:approximation:2}}

Let $X$ be a compact Riemannian $2$-manifold with boundary, equipped with a Riemannian metric $g$ that gives rise to a length metric $d_g$ on $X$. Let $d_X$ be a metric on $X$ and suppose that the identity map from $(X,d_g)$ to $(X,d_X)$ is an $\eta$-quasisymmetric homeomorphism (as we may suppose). We invite the reader to study the proof in parallel with Figure \ref{figure:sketch}.

\subsubsection*{Simplicial approximation of the smooth manifold $X$}
By Theorem \ref{theorem:bowditch}, for each sufficiently small $t>0$ there exists a metric simplicial complex $(Z,d_Z)$ whose $2$-simplices are copies of an equilateral triangle of side length $t$ and there exists a uniformly bi-Lipschitz homeomorphism $\tau \colon (Z,d_Z)\to (X,d_g)$. Moreover, by Remark \ref{remark:degree}, the graph $(V,\sim)$ of the $1$-skeleton of $Z$ has uniformly bounded degree. Thus, $(Z,d_Z)$ is a uniformly quasiconformal simplicial complex. 

\subsubsection*{The simplices of $Z$}
Let $S$ be a $2$-simplex of $Z$ with vertices $p_1, p_{2},p_{3}$. For $i\in \{1,2,3\}$ we define
$$d_i=\alpha \cdot d_X(\tau(p_{i}),\tau(p_{{i+1}})),$$
where $p_{4}\coloneqq p_{1}$ and $\alpha>0$ is a parameter to be specified. By Lemma \ref{lemma:simplicial:lower} we have $d_Z(p_i,p_j)\simeq t$ for $i,j\in \{1,2,3\}$. Since $\tau\colon (Z,d_Z)\to (X,d_g)$ is uniformly bi-Lipschitz and the identity map from $(X,d_g)$ to $(X,d_X)$ is $\eta$-quasisymmetric, we conclude from Lemma \ref{lemma:qs_heinonen} that 
$$d_i\simeq_{\eta} d_j$$
for $i,j\in \{1,2,3\}$. By Lemma \ref{lemma:triangle} there exist a $C(\eta)$-quasiconformal simplicial complex $K_S$ consisting of three $2$-simplices and a uniformly quasiconformal simplicial complex $S'$, which arises by connecting the barycenter of the equilateral triangle $S$ to its vertices, with the following properties: there exists a simplicial isomorphism from $S'$ onto $K_S$ so that the edge of $S'$ between $p_{i}$ and $p_{{i+1}}$ is mapped to an edge of $K_S$ of length $d_i$, $i\in \{1,2,3\}$. See Figure \ref{figure:lemma:triangle} for an illustration.

The subdivision of each $2$-simplex $S$ of $Z$ into three $2$-simplices gives rise to a metric simplicial complex $(Z',d_{Z'})$ that is isometric to $(Z,d_Z)$. Moreover, since the $1$-skeleton of $Z$ has bounded degree, we conclude that $(Z',d_{Z'})$ has the same property and is a uniformly quasiconformal simplicial complex.

\subsubsection*{Construction of a simplicial complex $Y$}
Whenever two simplicial complexes $S_1',S_2'$ (as above) of $Z'$ share an edge we glue isometrically the simplicial complexes $K_{S_1}$ and $K_{S_2}$ along the corresponding edges. The orientation of the gluing is specified by the simplicial isomorphisms between $S_i'$ and $K_{S_i}$, $i=1,2$. This process gives rise to a metric simplicial complex $(Y,d_Y)$. Equivalently one may think of ``replacing" each equilateral triangle $S$ of $Z$ with the simplicial complex $K_S$. The simplicial complex $Y$ is $C(\eta)$-quasiconformal.

\subsubsection*{The isomorphism between $Z'$ and $Y$}
Consider the simplicial map $f\colon (Z',d_{Z'})\to (Y,d_Y)$ arising by gluing the simplicial isomorphisms from  each simplicial complex $S'$ of $Z'$ onto the corresponding complex $K_{S}$. We first prove some distortion estimates based on the construction of the simplicial complexes.

Let $T\subset Y$ be a $2$-simplex and $f(p),f(q)$ be vertices of $T$. By Lemma \ref{lemma:simplicial:lower}, $\diam_{d_{Z'}}(f^{-1}(T))\simeq d_{Z'}(p,q)$. Since $\tau\colon (Z',d_{Z'})\to (X,d_X)$ is a composition of a uniformly bi-Lipschitz and an $\eta$-quasisymmetric map, there exists a constant $M_1\geq 1$ that depends only on $\eta$ such that
\begin{align}\label{main:x}
M_1^{-1}\leq  \frac{d_X(\tau(p),\tau(q))}{\diam_{d_X}(\tau (f^{-1}(T)))}\leq M_1.
\end{align} 

Next, let $x\in Y$ be a vertex that lies on a $2$-simplex $T\subset Y$. By construction, $T$ is contained in a simplicial complex $K_S=f(S')$ and one side of $T$ with endpoints $f(p),f(q)$ has length equal to $\alpha d_X(\tau(p),\tau(q))$. By Lemma \ref{lemma:simplicial:lower}, there exists a constant $M_2\geq 1$ depending only on $\eta$ such that
\begin{align}\label{main:epsilon_dt}
M_2^{-1}\leq \frac{\varepsilon(x)}{\diam_{d_T}(T)}\leq M_2
\end{align}
Moreover, by the definition of a quasiconformal simplicial complex, there exists a constant $M_3\geq 1$ depending only on $\eta$ such that 
\begin{align}\label{main:dt_dx}
M_3^{-1}\leq \frac{\diam_{d_T}(T)}{\alpha d_X(\tau(p),\tau(q))}\leq M_3.
\end{align}

\subsubsection*{Approximations of the metric spaces $(Z',d_{Z'})$ and $(Y,d_Y)$}

We start verifying the assumptions of Theorem \ref{theorem:approximation_quasisymmetric} for the map $f\colon (Z',d_{Z'})\to (Y,d_Y)$, which will allow us to conclude that $f$ is quasisymmetric after modifying appropriately the metric of $Y$.

First, note that both spaces $(Z',d_{Z'}), (Y,d_Y)$ are length spaces and in particular they have uniformly bounded turning as required in Theorem \ref{theorem:approximation_quasisymmetric} \ref{tqs:bt}. 

By Lemma \ref{lemma:simplicial:approximation}, there is a uniform constant $L\geq 1$ and a natural $(L,L)$-approxi\-ma\-tion $\mathcal A=((V,\sim),\mathfrak p,\mathfrak r,\mathcal U)$ of $(Z',d_{Z'})$. Specifically, $(V,\sim)$ is the graph of the $1$-skeleton of $Z'$, $\mathfrak p$ is a bijection between $V$ and the vertices of the $1$-skeleton of $Z'$, $\mathfrak r(v)=\varepsilon (\mathfrak{p}(v))$, and $\mathcal U(v)=\inter( \bigcup S(\mathfrak p(v)))$ for $v\in V$. Analogously, by the same lemma, there is an $(L,C(\eta))$-approximation $\mathcal A'=((V,\sim),\mathfrak p',\mathfrak r',\mathcal U')$ of $(Y,d_Y)$. The approximations $\mathcal A$ and $\mathcal A'$ are fine, as long as the side length $t$ of the triangles of $Z$ is sufficiently small.

By Lemma \ref{lemma:simplicial:lower}, for any two distinct vertices $u,v\in V$ we have $d_{Z'}(p_u,p_v)\geq \varepsilon(p_u)=r_u$. This verifies condition \ref{tqs:sep} of Theorem \ref{theorem:approximation_quasisymmetric}. By Lemma \ref{lemma:simplicial:approximation:isomorphism}, for each $v\in V$ the map $f|_{\st_{2L+1}(v)}\colon (\st_{2L+1}(v),d_{Z'})\to (f(\st_{2L+1}(v)), d_Y)$ is $C(\eta)$-quasisymmetric. Therefore, assumption \ref{tqs:star} of Theorem \ref{theorem:approximation_quasisymmetric} holds.

\subsubsection*{The vertex set of $Y$}
It remains to verify Theorem \ref{theorem:approximation_quasisymmetric} \ref{tqs:vert} for a suitable metric $d_{\mathcal S}$ on the set vertex $\mathcal S=f(\mathfrak p(V))$ of $Y$. For $u,v\in V$ we define 
$$d_{\mathcal S}(f(p_u),f(p_v))= d_X( \tau(p_u),\tau(p_v)).$$
Since $\tau\colon (Z',d_{Z'})\to (X,d_g)$ is uniformly bi-Lipschitz and the identity map from $(X,d_g)$ to $(X,d_X)$ is $\eta$-quasisymmetric, we conclude that $f\colon (\mathfrak p(V), d_{Z'})\to (\mathcal S, d_{\mathcal S})$ is $C(\eta)$-quasisymmetric. Finally, we show the inequalities that are required in Theorem \ref{theorem:approximation_quasisymmetric} \ref{tqs:vert} between $d_Y$ and $d_{\mathcal S}$.

Suppose that $u,v\in V$ and $u\sim v$. Then $f(p_u),f(p_v)$ are vertices of a $2$-simplex $T\subset Y$. By \eqref{main:dt_dx} we have
\begin{align*}
d_Y( f(p_u),f(p_v))&\leq \diam_{d_Y}(T) \leq \diam_{d_T}(T) \leq M_3 \alpha d_X(\tau(p_u),\tau(p_v))\\
&=M_3\alpha d_{\mathcal S}(f(p_u),f(p_v)).
\end{align*}
This gives the ultimate inequality appearing in Theorem \ref{theorem:approximation_quasisymmetric} \ref{tqs:vert}.

It remains to ensure that $d_{\mathcal S}\leq d_Y$. Let $u,v\in V$ with $k(u,v)\geq L$. We apply Lemma \ref{lemma:quasiconvex} to $(Y,d_Y)$ with the $(L,C(\eta))$-approximation $\mathcal A'$. As a consequence, there exist vertices $u=w_0,w_1,\dots,w_n=v$, $n\geq 0$, with $w_{i-1}\sim w_i$ for each $i\in \{1,\dots,n\}$, such that
$$\sum_{i=0}^n r_{w_i}' \leq M_4 d_Y(f(p_u),f(p_v)),$$
where $M_4$ is a positive constant that depends only on $\eta$. For $i\in \{1,\dots,n\}$, let $S_i$ be a $2$-simplex of $Z'$ that contains $p_{w_{i-1}}$ and $p_{w_i}$. By \eqref{main:x}, \eqref{main:dt_dx}, and \eqref{main:epsilon_dt}, we have
\begin{align*}
d_{\mathcal S}(f(p_u),f(p_v)) &= d_X(\tau(p_u),\tau(p_v))\leq \sum_{i=0}^n \diam_{d_X}(\tau (S_i))\\
&\leq \alpha^{-1}M_1M_2M_3 \sum_{i=0}^n \varepsilon(f(p_{w_i}))=\alpha^{-1}M_1M_2M_3\sum_{i=0}^n r_{w_i}'\\
&\leq \alpha^{-1}M_1M_2M_3M_4 d_Y(f(p_u),f(p_v)).
\end{align*}
Upon choosing $\alpha \geq M_1M_2M_3M_4$ we have $d_{\mathcal S}(f(p_u),f(p_v))\leq d_Y(f(p_u),f(p_v))$. 

Next, suppose that $k(u,v)<L$ and that $u\neq v$. By Lemma \ref{lemma:simplicial:lower} we have 
\begin{align}\label{main:dy_ru}
d_Y(f(p_u),f(p_v))\geq \varepsilon (f(p_u))=r_u'.
\end{align}
Now, consider vertices $u=w_0,w_1,\dots,w_n=v$, $1\leq n<L$, with $w_{i-1}\sim w_i$ for each $i\in \{1,\dots,n\}$. Arguing as above, we have
\begin{align*}
d_{\mathcal S}(f(p_u),f(p_v))=d_X(\tau(p_u),\tau(p_v))\leq \alpha^{-1}M_1M_2M_3 \sum_{i=0}^n r_{w_i}'.
\end{align*}
By property \ref{a:3} of an approximation (see Section \ref{section:approximations:definition}), there exists a constant $M_5\geq 1$ depending only on $\eta$ such that  $r_{w_i}'\leq M_5 r_u'$ for each $i\in \{0,\dots,n\}$. Therefore, by \eqref{main:dy_ru} we have
\begin{align*}
d_{\mathcal S}(f(p_u),f(p_v)) \leq \alpha^{-1}M_1M_2M_3 L M_5 r_u' \leq \alpha^{-1}M_1M_2M_3 L M_5 d_Y(f(p_u),f(p_v)).
\end{align*}
We finally choose $\alpha=M_1M_2M_3\max\{M_4, LM_5\}$ and we have completed the verification of condition \ref{tqs:vert} of Theorem \ref{theorem:approximation_quasisymmetric}.

\subsubsection*{Smoothing of the simplicial complex $Y$}
In order to smoothen the simplicial complex $Y$, we will use the result of Cattalani, Theorem \ref{theorem:cattalani}. This yields a Riemannian metric $h$ on $Y$ and the corresponding intrinsic metric $d_h$ such that
\begin{align}\label{main:a}
A^{-1}d_h\leq d_Y\leq d_h
\end{align}
for some constant $A\geq 1$ that depends only on $\eta$. We have verified that the map $f\colon (Z',d_{Z'})\to (Y,d_Y)$ satisfies the assumptions of Theorem \ref{theorem:approximation_quasisymmetric}. It is immediate that assumptions \ref{tqs:bt}, \ref{tqs:sep}, and \ref{tqs:star} of the theorem are satisfied also for the map $f\colon (Z',d_{Z'})\to (Y,d_h)$. Since $d_Y\simeq_\eta d_h$ and $d_Y\leq d_h$, we see that assumption \ref{tqs:vert} is also true. 

\subsubsection*{Construction of the space $(Y,\widetilde d)$}
We consider the metric $\widetilde d$ on $Y$ arising by gluing $d_h$ with $d_{\mathcal S}$ as in Lemma \ref{lemma:glue}. Then $(Y,\widetilde d)$ is locally isometric to the length space $(Y,d_h)$, in general we have $\widetilde d\leq d_h$, and the metric $\widetilde d$ agrees with $d_{\mathcal S}$ on the vertex set $f(\mathfrak p(V))$. By Theorem \ref{theorem:approximation_quasisymmetric}, the map $f\colon (Z',d_{Z'})\to (Y,\widetilde d)$ is $\eta'$-quasisymmetric for some distortion function $\eta'$ that depends only on $\eta$. The composition $f\circ \tau^{-1}\colon (X,d_g)\to (Y,\widetilde d)$ is $\widetilde \eta$-quasisymmetric for some distortion function $\widetilde \eta$ that depends only on $\eta$. 

\subsubsection*{Gromov--Hausdorff distance to $(X,d_X)$}
In order to complete the proof of Theorem \ref{theorem:main:approximation}, we will demonstrate that the map $\phi=\tau\circ f^{-1}\colon (Y,\widetilde d)\to (X,d_X)$ is an $\varepsilon$-isometry for an appropriate small $\varepsilon$. Note that if $T$ is a $2$-simplex of $Y$, by \eqref{main:a}, \eqref{main:dt_dx}, and \eqref{main:x} we have
\begin{align}\label{main:tilde_dx}
\diam_{\widetilde d}(T) &\leq \diam_{d_h}(T)\leq A\diam_{d_Y}(T)\\\notag
&\leq A\diam_{d_T}(T)\leq A\alpha M_1M_3 \diam_{d_X}(\phi(T)).
\end{align}
Since the identity map from $(X,d_g)$ to $(X,d_X)$ is $\eta$-quasisymmetric, by Lemma \ref{lemma:qs_heinonen} we have
\begin{align}\label{main:dx_dg}
\diam_{d_X}(\phi(T)) \leq \diam_{d_X}(X)\cdot  \eta \left( \frac{2\diam_{d_g}(\phi(T))}{\diam_{d_g}(X)}\right).
\end{align}
Since $\tau\colon (Z',d_{Z'})\to (X,d_g)$ is uniformly bi-Lipschitz, we have
\begin{align}\label{main:dg_dz}
\diam_{d_g}(\phi(T))\leq C \diam_{d_{Z'}}(f^{-1}(T)).
\end{align}
Finally, by construction, the $2$-simplices of $Z$ have side length $t$ and the $2$-simplices of $Z'$ have diameter bounded above by $t$, where $t$ is fixed at the beginning of the proof. Therefore, combining \eqref{main:tilde_dx}, \eqref{main:dx_dg}, and \eqref{main:dg_dz}, we have
\begin{align}\label{main:t_bound}
\max\{\diam_{\widetilde d}(T), \diam_{d_X}(\phi(T))\}\leq C_1 \eta(C_2t) 
\end{align}
where $C_1,C_2$ are positive constants that depend on $\eta$, $d_g$, and $d_X$, but not on $t$.

Now, let $y_1,y_2\in Y$ and suppose that they lie in $2$-simplices $T_1,T_2\subset Y$, respectively. Let $f(p_{u})$ be a vertex of $T_1$ and $f(p_v)$ be a vertex of $T_2$. We have
\begin{align*}
\widetilde d(y_1,y_2)&\leq \widetilde d(y_1,f(p_u)) +\widetilde d(f(p_u),f(p_v)) +\widetilde d(f(p_v),y_2)\\
&\leq \diam_{\widetilde d}(T_1)+ d_{\mathcal S}(f(p_u),f(p_v))+ \diam_{\widetilde d}(T_2)\\
&\leq 2C_1\eta(C_2t)+ d_X(\tau(p_u),\tau(p_v))\\
&\leq 2C_1\eta(C_2t)+ d_X(\tau(p_u), \phi(y_1)) +d_X(\phi(y_1),\phi(y_2)) + d_X(\phi(y_2),\tau(p_v))\\
&\leq 2C_1\eta(C_2t)+ \diam_{d_X}(\phi(T_1)) +d_X(\phi(y_1),\phi(y_2)) + \diam_{d_X}(\phi(T_2))\\
&\leq 4C_1\eta(C_2t)+ d_X(\phi(y_1),\phi(y_2)).
\end{align*}
Similarly, 
$$d_X(\phi(y_1),\phi(y_2))\leq 4C_1\eta(C_2t)+\widetilde d(y_1,y_2).$$
We have shown that $\phi$ is an $\varepsilon$-isometry for $\varepsilon=4C_1\eta(C_2t)$. Observe that $\varepsilon\to 0$ as $t\to 0$. This shows that the space $(Y,\widetilde d)$ can be taken as close to $(X,d_g)$ as we wish in the Gromov--Hausdorff distance. This completes the proof of Theorem \ref{theorem:main:approximation}.\qed

\subsection{Proof of Theorem \ref{theorem:main:approximation:bilip}}
We use the exact same notation and construction from the previous proof. We show that if the identity map from $(X,d_g)$ onto $(X,d_X)$ is $\lambda$-bi-Lipschitz for some $\lambda\geq 1$, then the map $f\circ \tau^{-1}\colon (X,d_g)\to (Y,\widetilde d)$ is $C(\lambda)$-bi-Lipschitz. 

Note that the identity map from $(X,d_g)$ onto $(X,d_X)$ is $\eta$-quasisymmetric for $\eta(t)=\lambda^2t$. Since $\tau \colon (Z,d_Z)\to (X,d_X)$ is $C(\lambda)$-bi-Lipschitz, for each $2$-simplex of $Z$ of side length $t$ and with vertices $p_1,p_2,p_3$, the distances $d_X(\tau(p_i),\tau(p_{i+1}))$, $i=1,2,3$, are comparable to $t$, depending only on $\lambda$. Thus, the constructed space $(Y,d_Y)$ is a $C(\lambda)$-quasiconformal simplicial complex with the additional property that each triangle of $Y$ can be mapped to an equilateral triangle of side length $t$ with a $C(\lambda)$-bi-Lipschitz linear map. We conclude that the simplicial map $f\colon (Z',d_{Z'})\to (Y,d_Y)$, restricted to each $2$-simplex of $Z'$, is linear and $C(\lambda)$-bi-Lipschitz. Thus, $f\colon (Z',d_{Z'})\to (Y,d_Y)$ is $C(\lambda)$-bi-Lipschitz; recall the definition of the metric of a simplicial complex from Section \ref{section:simplicial}. As a consequence of the above and \eqref{main:a}, the map $f\circ \tau^{-1}\colon (X,d_g)\to (Y,d_h)$ is $C(\lambda)$-bi-Lipschitz. 

It remains to show that the identity map from $(Y,d_h)$ onto $(Y,\widetilde d)$ is bi-Lipschitz. The metric $\widetilde d$ arises by gluing $d_h$ with $d_{\mathcal S}$ as in Lemma \ref{lemma:glue}. Recall that $\mathcal S$ is the vertex set of $Y$ and by definition, for $p,q\in \mathcal S$ we have
$$d_{\mathcal S}(p,q) = d_X( \tau(f^{-1}(p)),\tau(f^{-1}(q))).$$
Since the maps $id\colon (X,d_g)\to (X,d_X)$ and $f\circ \tau^{-1}\colon (X,d_g)\to (Y,d_h)$ are $C(\lambda)$-bi-Lipschitz, we have
$$d_{\mathcal S}(p,q) \simeq_{\lambda} d_{g}( \tau(f^{-1}(p)),\tau(f^{-1}(q))) \simeq_{\lambda} d_h(p,q).$$
By Lemma \ref{lemma:glue} \ref{g:lambda}, the identity map from $(Y,d_h)$ onto $(Y,\widetilde d)$ is $C(\lambda)$-bi-Lipschitz, as desired.\qed

\section{Metric spheres of finite area}\label{section:finite_area}

This section contains the proofs of Theorem \ref{theorem:bac}, Theorem \ref{theorem:reciprocal}, and Theorem \ref{theorem:qc_qs}. The material is independent of the rest of the paper.

\subsection{Geometric notions}\label{section:geometric_notions}
Let $(X,d,\mu)$ be a metric measure space. For a set $G\subset X$ and disjoint sets $E,F\subset G$ we define $\Gamma(E,F;G)$ to be the family of curves in $G$ joining $E$ and $F$. Let $\Gamma$ be a family of curves in $X$. A Borel function $\rho\colon X\to [0,\infty]$ is \textit{admissible} for $\Gamma$ if $\int_\gamma \rho\, ds\geq 1$ for all rectifiable paths $\gamma\in \Gamma$. Let $Q\geq 1$. We define the \textit{$Q$-modulus} of $\Gamma$ as
$$\Mod_Q \Gamma=\inf_\rho \int_X \rho^Q \, d\mu,$$
where the infimum is taken over all admissible functions $\rho$ for $\Gamma$. We say that $X$ is a \textit{$Q$-Loewner space} if there exists a decreasing function $\phi\colon (0,\infty)\to (0,\infty)$ such that for every pair of disjoint non-degenerate continua $E,F\subset X$ we have
$$\Mod_Q\Gamma(E,F;X)\geq \phi(\Delta(E,F)),$$
where $\Delta(E,F)$ is the \textit{relative distance} of $E$ and $F$, which is defined by 
$$ \frac{\dist(E,F)}{\min\{\diam (E),\diam(F)\}}.$$
We say that $X$ is \textit{Ahlfors $Q$-regular} if there exists a constant $M\geq 1$ such that 
$$M^{-1}r^Q\leq \mu(B(x,r))\leq Mr^Q$$
for each $x\in X$ and $0<r<\diam (X)$.

Let $(X,d)$ be a metric space. We say that $X$ is \textit{linearly locally connected} ($\llc$) if there exists a constant $M\geq 1$ such that for each ball $B(a,r)$ in $X$ the following two conditions hold.
\begin{enumerate}[label=\normalfont(LLC$_{\arabic*}$)]
\item\label{llc1} For every $x,y\in B(a,r)$ there exists a connected set $E\subset B(a,Mr)$ that contains $x$ and $y$.
\item\label{llc2} For every $x,y\in X\setminus B(a,r)$ there exists a connected set $E\subset X\setminus B(a,r/M)$ that contains $x$ and $y$.
\end{enumerate}
In that case we say that $X$ is $M$-LLC. We say that $X$ is \textit{doubling} if there exists a constant $M\geq 1$ such that for every $R>0$, every ball of radius $R$ can be covered by at most $M$ balls of radius radius $R/2$. In that case we say that $X$ is $M$-doubling.

\subsection{Quasiconformal maps}
A \textit{metric surface} is a $2$-dimensional topological manifold with a metric inducing its topology. In this section, metric surfaces are always considered as metric measure spaces endowed with the Hausdorff $2$-measure (or else area measure). Let $X,Y$ be metric surfaces of locally finite Hausdorff $2$-measure. A homeomorphism $h\colon X\to Y$ is \textit{quasiconformal} if there exists $K\geq 1$ such that 
$$K^{-1} \Mod_2\Gamma\leq \Mod_2 h(\Gamma)\leq K\Mod_2\Gamma$$
for every curve family $\Gamma$ in $X$. In this case we say that $h$ is $K$-quasiconformal. We say that a map $h\colon X\to Y$ is \textit{weakly quasiconformal} if $h$ is the uniform limit of a sequence of homeomorphisms and  there exists $K>0$ such that for every curve family $\Gamma$ in $X$ we have
\begin{align*}
\Mod_2 \Gamma\leq K \Mod_2 h(\Gamma).
\end{align*}
In this case we say that $h$ is weakly $K$-quasiconformal.


Rajala \cite{Rajala:uniformization} introduced the notion of a \textit{reciprocal} metric surface, which we do not define here for the sake of brevity, and proved that a metric $2$-sphere $X$ of finite area is quasiconformally equivalent to the Riemann sphere $\widehat \C$ if and only if $X$ is reciprocal. 

More generally, it was shown in \cite{NtalampekosRomney:nonlength} that any metric surface of locally finite Hausdorff $2$-measure admits a weakly quasiconformal parametrization by a Riemannian surface of the same topological type. The following special case is sufficient for our purposes.

\begin{theorem}[\cite{NtalampekosRomney:nonlength}*{Theorem 1.2}]\label{theorem:wqc}
    Let $X$ be a metric $2$-sphere of finite Hausdorff $2$-measure. Then there exists a weakly $(4/\pi)$-quasiconformal map from $\widehat \C$ onto $X$. 
\end{theorem}

When the target surface is reciprocal, a weakly quasiconformal map is upgraded to a genuine quasiconformal map, according to the next lemma.

\begin{lemma}[\cite{MeierNtalampekos:rigidity}*{Lemma 2.11}]\label{lemma:wqc_qc}
Let $X,Y$ be metric $2$-spheres with finite Hausdorff 2-measure such that $Y$ is reciprocal. Then every weakly quasiconformal map $f\colon X \to Y$ is a quasiconformal homeomorphism, quantitatively.
\end{lemma}

Tyson \cite{Tyson:lusin} proved that quasisymmetric maps from Euclidean space into metric surfaces of finite area are weakly quasiconformal. We state a special case of that result.
\begin{theorem}[\cite{Tyson:lusin}*{Theorem 3.13}]\label{theorem:tyson}
Let $X$ be a metric $2$-sphere of finite Hausdorff $2$-measure. Then every quasisymmetric map from $\widehat \C$ onto $X$ is weakly quasiconformal, quantitatively.
\end{theorem}

Conversely, weakly quasiconformal maps can be upgraded to quasisymmetric maps if one imposes some geometric assumptions on the spaces, namely the ones appearing in the statement of Theorem \ref{theorem:smooth}.

\begin{theorem}\label{theorem:upgrade}
Let $X,Y$ be metric $2$-spheres of finite Hausdorff $2$-measure and let $f\colon X\to Y$ be a weakly quasiconformal map. For $i\in \{1,2,3\}$ let $x_i\in X$ and suppose there exists $L\geq 1$ such that that for $i\neq j$ we have
\begin{align*}
{d_X(x_i,x_j)}\geq L^{-1}\diam (X) \quad \text{and} \quad d_Y(f(x_i), f(x_j)) \geq L^{-1}{\diam (Y)}.
\end{align*}  
If $X$ satisfies \ref{smooth:loewner} and $Y$ satisfies \ref{smooth:llc-mod}, then $f$ is a quasisymmetric homeomorphism.
\end{theorem}

To prove this, we use the following consequence of \ref{smooth:llc-mod}\ref{smooth:modulus}.

\begin{lemma}\label{lemma:modulus_log}
Let $Y$ be a metric surface of locally finite Hausdorff $2$-measure. Suppose that there exist constants $L>1$ and $M>0$ such that for every ball $B(a,r)\subset X$ we have 
$$\Mod_2 \Gamma(\br B(a,r), Y\setminus B(a,L r);Y) <M.$$
Then for each $a\in Y$, $r>0$ and $R\geq Lr$ we have
$$\Mod_2\Gamma(\br B(a,r) , Y\setminus B(a,R);Y)\leq C(L,M) \left(\log \frac{R}{r}\right)^{-1}.$$
\end{lemma}
\begin{proof}
Let $N\in \N$ such that $L^Nr\leq R$ and $L^{N+1}r>R$. By assumption we have
$$\Mod_2 \Gamma_k<M,$$
where $\Gamma_k=\Gamma( \br B(y,L^{-k}R) , Y\setminus B(y, L^{-k+1}R); Y)<M$
for $k\in \{1,\dots,N\}$. For $k\in \{1,\dots,N\}$, let $\rho_k\colon Y\to[0,\infty]$ be a Borel function that is admissible for $\Gamma_k$ such that $\|\rho\|_{L^2(Y)}^2<M$. Let $A_k=B(y,L^{-k+1}R)\setminus \br B(y,L^{-k}R)$ and define 
$$\rho= \frac{1}{N}\sum_{k=1}^N \rho_k\chi_{A_k}.$$
Each curve $\gamma\in \Gamma=\Gamma(\br B(a,r) , Y\setminus B(a,R);Y)$ has a subcurve $\gamma_k\in \Gamma_k$ whose trace lies in $A_k$ except for the endpoints. Thus, $\int_{\gamma_k}\rho_k\chi_{A_k}\, ds\geq 1$ for $k\in \{1,\dots,N\}$ and
$$\int_{\gamma}{\rho}\, ds \geq \frac{1}{N} \sum_{k=1}^N \int_{\gamma_k}\rho_k\chi_{A_k}\, ds\geq 1.$$
This shows that $\rho$ is admissible for $\Gamma$ and
\begin{align*}
\Mod_2\Gamma \leq \int \rho^2\, d\mathcal H^2= \frac{1}{N^2} \sum_{k=1}^N \int_{A_k} \rho_k^2\, d\mathcal H^2 \leq \frac{M}{N}\leq 2M\log L \left( \log \frac{R}{r}\right)^{-1}.
\end{align*}
This completes the proof. 
\end{proof}

\begin{proof}[Proof of Theorem \ref{theorem:upgrade}]
We first show that $f$ is a homeomorphism. By \ref{smooth:llc-mod}\ref{smooth:modulus}  and Lemma \ref{lemma:modulus_log}, for each $y\in Y$ and $\delta>0$, the $2$-modulus of the family of curves with diameter at least $\delta$ that pass through $y$ is zero. By the subadditivity of modulus, for each $y\in Y$ the family of non-constant curves passing through $y$ has $2$-modulus zero. Now, based on this property, \cite{NtalampekosRomney:length}*{Theorem 7.4} implies that the weakly quasiconformal map $f$ must be a homeomorphism.

The argument to show that $f$ is quasisymmetric is quite standard and can be found in \cite{LytchakWenger:parametrizations}*{Theorem 2.5} or \cite{Rajala:uniformization}*{Proof of Corollary 1.7}, so we omit it and we restrict ourselves to some comments. In these results the space $Y$ is assumed to be Ahlfors $2$-regular and condition \ref{smooth:llc-mod}\ref{smooth:modulus} is derived as a consequence. In our setting, the three-point normalizations on $X$ and $Y$, the assumption that $X$ is $2$-Loewner, and the assumptions that $Y$ is $\llc$ as in \ref{smooth:llc-mod}\ref{smooth:llc} and satisfies \ref{smooth:llc-mod}\ref{smooth:modulus} (or equivalently the conclusion of Lemma \ref{lemma:modulus_log}) imply that the weakly quasiconformal homeomorphism $f$ is \textit{weakly quasisymmetric} (see \cite{Heinonen:metric}*{Section 10} for the definition). Then the doubling assumptions on $X$ and $Y$ imply that $f$ is actually quasisymmetric. 
\end{proof}

\begin{proof}[Proof of Theorem \ref{theorem:bac}]
Let $X$ be a metric $2$-sphere of finite Hausdorff $2$-measure. Suppose that \ref{smooth:llc-mod} is true. By Theorem \ref{theorem:wqc} there exists a weakly $(4/\pi)$-quasicon\-formal map $h\colon \widehat\C\to X$. By Theorem \ref{theorem:upgrade}, $h$ is a homeomorphism. We fix three points $x_i\in X$, $i\in \{1,2,3\}$, with $d_X(x_i,x_j)\geq \diam(X)/2$ for $i\neq j$. We may precompose $h$ with a M\"obius transformation of $\widehat{\C}$ so that the preimages $h^{-1}(x_i)$, $i\in \{1,2,3\}$, have mutual distances at least $\diam (\widehat{\C})/2$ for $i\neq j$. Note that $\widehat \C$ is a doubling and $2$-Loewner space \cite{Heinonen:metric}*{Chapter 8}, so it satisfies \ref{smooth:loewner}. By Theorem \ref{theorem:upgrade}, the normalized map $h$ is quasisymmetric, quantitatively. This proves that \ref{smooth:quasisphere} is true. 

Next, we assume that there exists an $\eta$-quasisymmetric map $h\colon \widehat \C\to X$ for some distortion function $\eta$ as in \ref{smooth:quasisphere}. This implies that $X$ is doubling \cite{Heinonen:metric}*{Theorem 10.18}, quantitatively. We show that $X$ is a $2$-Loewner space, quantitatively. By Theorem \ref{theorem:tyson}, $h$ is weakly $K$-quasiconformal, where $K$ depends only on $\eta$. Let $E',F'\subset X$ be disjoint continua and set $E=h^{-1}(E')$, $F=h^{-1}(F')$. Since $\widehat \C$ is a $2$-Loewner space, there exists a decreasing function $\phi\colon (0,\infty)\to (0,\infty)$ such that 
$$ \Mod_2 \Gamma(E,F;\widehat \C) \geq \phi(t),$$
where $t=\Delta(E,F)$. It is elementary to show from the definition of a quasisymmetric map that the relative distance $t'\coloneqq \Delta(E',F')$ satisfies $t\leq \eta'(2t')$, where $\eta'(s)=1/\eta^{-1}(s^{-1})$, $s>0$; see \cite{Rehmert:thesis}*{Proposition 2.1.9} for an argument. Therefore, 
$$ \Mod_2\Gamma(E',F';X) \geq K^{-1}\Mod_2 \Gamma(E,F;\widehat \C) \geq K^{-1}\phi(t) \geq K^{-1}\phi(\eta'(2t')).$$
This completes the proof that $X$ satisfies \ref{smooth:loewner}.
\end{proof}

For the proof of Theorem \ref{theorem:qc_qs} we will use the following result of Williams.

\begin{theorem}[\cite{Williams:qc}*{Theorem 1.2}] \label{theorem:williams}
Let $X,Y$ be separable metric measure spaces with locally finite Borel regular outer measures $\mu,\nu$, respectively, such that $\nu$ is doubling. Let $f\colon X\to Y$ be a homeomorphism and $Q>1$. The following are quantitatively equivalent. 
\begin{enumerate}[label=\normalfont(\roman*)]
	\item There exists $K>0$ such that for every family of curves $\Gamma$ in $X$ we have
	$$\Mod_Q\Gamma\leq K \Mod_Q f(\Gamma).$$
	\item\label{theorem:williams:ii}  There exist $L>1$ and $H>0$ such that for all $y\in Y$ we have
	$$\liminf_{r\to 0^+} \frac{\Mod_Q f^{-1}(\Gamma(\br B(y,r), Y\setminus B(y,Lr);Y))  }{\nu(B(y,r)) r^{-Q}} \leq H.$$
\end{enumerate}
\end{theorem}

\begin{proof}[Proof of Theorem \ref{theorem:qc_qs}]
Let $h\colon \widehat{\C}\to X$ be an $\eta$-quasisymmetry for some distortion function $\eta$. Since $h^{-1}$ is $\eta'$-quasisymmetric, where $\eta'$ depends only on $\eta$, for each ball $B(a,r)\subset X$ there exist $r_1>0$ and $L_1>1$, where $\eta'(1/L_1)=1/2$, such that 
\begin{align}\label{theorem:qc_qs:qs_inverse}
h^{-1}(B(a,r)) \subset B(h^{-1}(a),r_1) \subset B(h^{-1}(a), 2r_1) \subset h^{-1}(B(a,L_1r))
\end{align}
and given $L>1$, there exist exist $r_2>0$ and $L_2>1$, where $L_2=\eta'(L)$, such that 
\begin{align}\label{theorem:qc_qs:qs_inverse2}
B(h^{-1}(a),r_2) \subset h^{-1}(B(a,r)) \subset h^{-1}(B(a,Lr)) \subset B(h^{-1}(a), L_2r_2).
\end{align}

Suppose that $h$ is $K$-quasiconformal for some $K\geq 1$ as in part \ref{qc_qs:1} of the theorem. Thus, for each ball $B(a,r)\subset X$, by \eqref{theorem:qc_qs:qs_inverse} there exists $L_1>0$ that depends only on $\eta$  and there exists $r_1>0$ such that
\begin{align*}
\Mod_2 \Gamma( \br B(a,r), & X\setminus B(a,L_1r);X)\\
&\leq K\Mod_2 \Gamma( h^{-1}(\br B(a,r), \widehat{\C}\setminus h^{-1}(B(a,L_1r)) ;\widehat{\C})\\
&\leq K \Mod_2 \Gamma( \br B(h^{-1}(a), r_1), \widehat{\C} \setminus B(h^{-1}(a),2r_1); \widehat{\C}) < C
\end{align*}
where the latter bound is uniform and follows from \cite{Heinonen:metric}*{Lemma 7.18}. This proves \ref{qc_qs:3}. Note that \ref{qc_qs:3} trivially implies \ref{qc_qs:2}.

Finally, assume that there exist constants $L>1$ and $M>0$ such that 
$$\liminf_{r\to 0^+}\Mod \Gamma(\br B(a,r), X\setminus B(a,L r);X) <M$$
for every $a\in X$, as in \ref{qc_qs:2}. By \eqref{theorem:qc_qs:qs_inverse2}, there exists $L_2>1$ that depends only on $\eta,L$ and there exists $r_2>0$  such that
\begin{align*}
\Mod_2 h(\Gamma( \br B(h^{-1}(a), r_2), &\widehat{\C}\setminus B(h^{-1}(a), L_2r_2);\widehat{\C})) \\
&\leq \Mod_2 h(\Gamma( h^{-1}(\br B(a,r)), \widehat{\C}\setminus h^{-1}(B(a,Lr)); \widehat{\C}))\\
&=\Mod_2  \Gamma(\br B(a,r), X\setminus B(a,L r);X).
\end{align*}
By \eqref{theorem:qc_qs:qs_inverse2} we note that $r_2\to 0$ as $r\to 0$. Taking liminf as $r\to 0$ shows that condition \ref{theorem:williams:ii} of Theorem \ref{theorem:williams} is satisfied for the map $h^{-1}$. Therefore, there exists $K>0$ such that
\begin{align*}
\Mod_2 h(\Gamma)\leq K\Mod_2 \Gamma 
\end{align*}
for every family of curves $\Gamma$ in $\widehat{\C}$. By Lemma \ref{lemma:wqc_qc} we conclude that $h$ is quasiconformal, quantitatively, as required in \ref{qc_qs:1}.
\end{proof}

\begin{proof}[Proof of Theorem \ref{theorem:reciprocal}]
Let $X$ be a reciprocal metric $2$-sphere. Suppose that $X$ is a quasisphere as in \ref{smooth:quasisphere}. We will show that \ref{smooth:llc-mod} is true. Note that $X$ is doubling \cite{Heinonen:metric}*{Theorem 10.18} and it is elementary to show from the definitions that a metric space that is quasisymmetric to an $\llc$ space is also $\llc$; see \cite{Rehmert:thesis}*{Proposition 2.1.7} for an argument. Let $h\colon \widehat{\C}\to X$ be a quasisymmetric homeomorphism. By Theorem \ref{theorem:tyson} and Lemma \ref{lemma:wqc_qc}, $h$ is quasiconformal, quantitatively. By Theorem \ref{theorem:qc_qs}, condition \ref{smooth:llc-mod}\ref{smooth:modulus} is true.

Next, suppose that $X$ is doubling and is a $2$-Loewner space, as in \ref{smooth:loewner}. Consider a quasiconformal map $h\colon \widehat{\C}\to X$, provided by Rajala's theorem \cite{Rajala:uniformization}. Note that the map $h^{-1}\colon X\to \widehat{\C}$ is a quasiconformal map from a space satisfying \ref{smooth:loewner} onto a space satisfying \ref{smooth:llc-mod}. We can normalize this map appropriately via postcomposition with a M\"obius transformation of $\widehat \C$, as required in the assumptions of Theorem \ref{theorem:upgrade}. As a consequence, $h^{-1}$ is quasisymmetric, quantitatively.
\end{proof}

\bibliography{../../biblio} 

\end{document}